\documentclass[12pt]{amsart}

\usepackage{amssymb,amsthm,amsmath}
\usepackage{fullpage}
\usepackage{paralist}
\usepackage{graphicx}
\usepackage{hyperref}

\theoremstyle{plain}
\newtheorem{thm}{Theorem}[section]
\newtheorem{lem}[thm]{Lemma}
\newtheorem{prop}[thm]{Proposition}
\newtheorem{cor}[thm]{Corollary}

\theoremstyle{definition}

\newtheorem{dfn}[thm]{Definition}

\theoremstyle{remark}
\newtheorem{rmk}[thm]{Remark}

\DeclareMathOperator{\Div}{\mathrm{Div}}
\DeclareMathOperator{\divf}{\mathrm{div}}
\DeclareMathOperator{\RDiv}{\mathbb{R}\mathrm{Div}}
\DeclareMathOperator{\DivPlus}{\mathrm{Div}_+}
\DeclareMathOperator{\DivPlusD}{\mathrm{Div}_+^d}

\DeclareMathOperator{\RDivPlus}{\mathbb{R}\mathrm{Div}_+}
\DeclareMathOperator{\RDivPlusD}{\mathbb{R}\mathrm{Div}_+^d}
\DeclareMathOperator{\mR}{\mathbb{R}}
\DeclareMathOperator{\mZ}{\mathbb{Z}}
\DeclareMathOperator{\normalize}{\mathcal{N}}
\DeclareMathOperator{\continuous}{\mathcal{C}(\Gamma)}
\DeclareMathOperator{\CPA}{\mathrm{CPA}(\Gamma)}
\DeclareMathOperator{\BDV}{\mathrm{BDV}(\Gamma)}
\DeclareMathOperator{\measure}{\mathrm{Meas}^0(\Gamma)}

\DeclareMathOperator{\dist}{\mathrm{dist}_\Gamma}

\DeclareMathOperator{\imag}{\mathrm{imag}}
\DeclareMathOperator{\tconv}{\mathrm{tconv}}

\DeclareMathOperator{\Gmin}{\Gamma_{\mathrm{min}}}
\DeclareMathOperator{\Gmax}{\Gamma_{\mathrm{max}}}
\DeclareMathOperator{\supp}{\mathrm{supp}}
\DeclareMathOperator{\Sfunc}{\mathcal{B}}

\title{Tropical Convexity and Canonical Projections}
\author{Ye Luo}
\address{School of Mathematics, Georgia Institute of Technology, Atlanta, GA 30332}
\email{luoye@math.gatech.edu}


\date{}
\begin{document}
\maketitle

\begin{abstract}
\noindent
Using a potential theory on metric graphs $\Gamma$, we introduce the notion of tropical convexity to the space $\RDivPlusD(\Gamma)$ of effective $\mR$-divisors of degree $d$ on $\Gamma$ and show that a natural metric can be defined on $\RDivPlusD(\Gamma)$.
In addition, we extend the notion of reduced divisors which is conventionally defined in a complete linear system $|D|$ with respect to a single point in $\Gamma$. In our general setting, a reduced divisor is defined uniquely as an $\mR$-divisor in a compact tropical convex set $T\subset\RDivPlusD(\Gamma)$ with respect to a certain $\mR$-divisor $E$ of the same degree $d$. In this sense, we consider reduced divisors as canonical projections onto $T$. We also investigate some basic properties of tropical convex sets using techniques developed from general reduced divisors.
\end{abstract}

\section{Introduction}

\subsection{Notations and terminologies}
Let $\Gamma$ be a compact metric graph with finite edge lengths. For simplicity, we also denote the set of points of $\Gamma$ by $\Gamma$. Let $\Div(\Gamma)$ be the free abelian group on $\Gamma$. Let $\RDiv(\Gamma) = \Div(\Gamma)\otimes\mR$. As in convention, we call the elements of $\Div(\Gamma)$ \emph{divisors} (or \emph{$\mZ$-divisors} when we want to emphasize the integer coefficients), and elements of $\RDiv(\Gamma)$ \emph{$\mR$-divisors}. In cases of no confusion, we may also call $\mR$-divisors just as divisors throughout this paper. Let $\DivPlus(\Gamma)$ and $\RDivPlus(\Gamma)$ be the semigroups of effective $\mZ$-divisors and effective $\mR$-divisors respectively. If $d$ is a nonnegative integer, denote the set of effective divisors of degree $d$ by $\DivPlusD(\Gamma)$. If $d$ is a nonnegative real, denote the set of effective $\mR$-divisors of degree $d$ by $\RDivPlusD(\Gamma)$.

For a continuous function $f$ on $\Gamma$. Let $\normalize(f)  = f-\min f$. Let $\Gmin(f):=f^{-1}(\min{f})=\{v\in\Gamma|f(v)=\min{f}\}$ and $\Gmax(f):=f^{-1}(\max{f})=\{v\in\Gamma|f(v)=\max{f}\}$. In other words, $\Gmin(f)$ and $\Gmax(f)$ are the minimizer and maximizer of $f$ respectively.

\subsection{Overview}
Since Baker and Norine found a graph-theoretic analogue of the famous Riemann-Roch theorem on algebraic curves in their ground-breaking paper \cite{BN07}, such RR-type theorems have been extended to other combinatorial and geometric settings, such as weighted graphs \cite{AC13}, metric graphs, tropical curves, \cite{GK07,MZ07}, lattices \cite{AM10}, and finite sets \cite{JM12}.

The notion of reduced divisors is a main tool used by Baker and Norine to prove the graph-theoretical Riemann-Roch theorem and it appears in subsequent works by different authors \cite{A12,CDPR10,HKN07,L11}. The notion of reduced divisors (under different names) and a related notion of chip-firing games were originally introduced in a self-organized sandpile model on grids and then on arbitrary graphs \cite{D90}, and have aroused interest in various fields of research (see the short survey article \cite{LP10}) including combinatorics, theoretical physics, and arithmetic geometry. In the context of metric graphs $\Gamma$, reduced divisors arise in the following way: for a complete linear system $|D|$ (a linearly equivalent component of $\DivPlus(\Gamma)$) and a point $q\in\Gamma$, there exists a canonically defined divisor $D_0\in|D|$ which is \emph{``reduced''} with respect to $q$.

There are several equivalent ways \cite{MZ07,PS04} to characterize reduced divisors. Recently, Baker and Shokrieh made its connection to potential theory on (metric) graphs. The main tool in their theory is the \emph{energy pairing}, and for a fixed $q\in\Gamma$, it can be used to define two functions on the divisor group, the energy function $\mathcal{E}_q$ and the $b$-function $b_q$. Then the reduced divisor in $|D|$ with respect to $q$ is the minimizer of either $\mathcal{E}_q$ or $b_q$. In this paper, we are particularly interested in $b$-functions and have made an extension in our settings. 

In \cite{HMY12}, the authors studied the linear systems using the conventional theory of tropical convexity \cite{DS04}. In this sense, complete linear systems are tropically convex. In this paper, we have also generalized the notion of tropical convexity. More specifically, we have developed a geometric foundation for the notion of tropical convexity in the space of all $\mR$-divisors. In particular, we have found a canonical metric structure on the space of divisors, which can be used to study the topology and geometry on it. The notion of tropical convexity is intrinsically built on this metric structure. In this sense, the linear systems $|D|$ are tropical-path-connected components of $\DivPlus(\Gamma)$.

With our extended notions of $b$-functions and tropical convexity, we are able to generalize the notion of reduced divisors in the following sense:
\begin{enumerate}
  \item Reduced divisors exist not only just for complete linear systems $|D|$ but also for any compact tropically convex subset of $\RDivPlusD(\Gamma)$ with a given $d$.
  \item Reduced divisors can be defined not only with respect to a point $p$ on the metric graph but also any divisor $E\in\RDivPlus(\Gamma)$.
\end{enumerate}

Using general reduced divisors, we further develop tools to investigate some basic properties of tropical convexity, e.g., the contractibility and compactness of tropical convex hulls. In addition, tropical projection maps are canonically derived from general reduced divisors.

The paper is structured as follows. The potential theory on metric graphs is briefly reviewed in Section~\ref{S:potential}. We then define a metric structure on $\RDivPlusD(\Gamma)$ and study the induced topology in Section~\ref{S:metricstructure}. Our settings of tropical convexity are discussed in Section~\ref{S:tconvset}, where we also make statements of some basic properties of tropical convex sets. We introduce the notion of general reduced divisors and provide several criterions in Section~\ref{S:GRD}. Then we investigated several particular cases about general reduced divisors on tropical segments and develop some useful tools in Section~\ref{S:RD_tseg}. As an application of these tools, the theorems about the basic properties of tropical convex sets (stated in Section~\ref{S:tconvset}) are proved in Section~\ref{S:RD_tconv}. Finally, we discuss canonical projections in Section~\ref{S:Canproj}.

\section{Potential theory on metric graphs} \label{S:potential}
We list here some standard terminologies and basic facts concerning potential theory on metric graphs. The reader may refer \cite{BF06,BR07} for details.

For a metric graph $\Gamma$, we let $\continuous$ be the $\mR$-algebra of continuous real-valued functions on $\Gamma$, and let $\CPA\subset\continuous$ be the vector space consisting of all continuous piecewise-affine (or piecewise-linear) functions on $\Gamma$. Note that $\CPA$ is dense in $\continuous$. Let $\measure$ be the vector space of finite signed Borel measures of total mass zero on $\Gamma$. Denote by $\mR\in\continuous$ the space of constant functions on $\Gamma$.

In terms of electric network theory, we may think of $\Gamma$ as an electrical network with resistances given by the edge lengths. For $p,q,x\in \Gamma$, we define a $j$-function $j_q(x,p)$ as the potential at $x$ when one unit of current enters the network at $p$ and exits at $q$ with $q$ grounded (potential $0$).

We have the following properties of the $j$-function.
\begin{enumerate}
\item $j_q(x,p)$ is jointly continuous in $p$, $q$ and $x$.
\item $j_q(x,p)\in\CPA$.
\item $j_q(q,p)=0$.
\item $0\leqslant j_q(x,p) \leqslant j_q(p,p)$.
\item $j_q(x,p)=j_q(p,x)$.
\item $j_q(x,p)+j_p(x,q)$ is constant for all $x\in\Gamma$. Denoted by $r(p,q)$, this constant is the effective resistance between $p$ and $q$.
\item $r(p,q)=j_q(p,p)=j_p(q,q)$.
\item $r(p,q)\leqslant\dist(p,q)$ where $\dist(p,q)$ is the distance between $p$ and $q$ on $\Gamma$.
\item $\frac{r(p,q)}{\dist(p,q)}\rightarrow 1$ as $\dist(p,q)\rightarrow 0$.
\end{enumerate}

Let $\BDV$ be the vector space of functions of bounded differential variation \cite{BR07}. Then we have $\CPA\subset\BDV\subset\continuous$.

The Laplacian $\Delta:\BDV\rightarrow\measure$ is defined as an operator in the following sense.
\begin{enumerate}
\item $\Delta$ induces an isomorphism between $\BDV/\mR$ and $\measure$ as vector spaces.
\item For $f\in\CPA$, we have $$\Delta f=\sum_{p\in\Gamma}\sigma_p(f)\delta_p$$ where $-\sigma_p(f)$ is the sum of the slopes of $f$ in all tangent directions emanating from $p$ and $\delta_p$ is the Dirac measure (unit point mass) at $p$. In particular, $\Delta j_q(x,p)=\delta_p(x)-\delta_q(x)$.
\item An inverse to $\Delta$ is given by $$\nu\mapsto\int_{\Gamma}j_q(x,y)d\nu(y)\in\{f\in\BDV:f(q)=0\}.$$
\end{enumerate}

\section{A metric structure defined on $\RDivPlusD(\Gamma)$} \label{S:metricstructure}
If $D=\sum_{p\in\Gamma}m_p\cdot(p)\in\RDiv$, we let $\delta_D:=\sum_{p\in\Gamma}m_p\cdot\delta_p$ with $\delta_p$ the Dirac measure at $p$.
Let $D_1,D_2\in\RDivPlusD(\Gamma)$. Then based on the potential theory on $\Gamma$, there exist a piecewise-linear function $f_{D_2-D_1}\in\CPA$ on $\Gamma$ such that $\Delta f_{D_2-D_1} = \delta_{D_2}-\delta_{D_1}$. Note that any two such associated functions differ in a constant. In this sense, we say $\divf(f):=D_2-D_1$ is the \emph{associated divisor} of $f_{D_2-D_1}$, and correspondingly $f_{D_2-D_1}$ is an \emph{associated function} of $D_2-D_1$. Then $\normalize(f_{D_2-D_1})$ has minimum $0$ and is unique with $D_1$ and $D_2$ provided.

More precisely, if $D_1=(q)$ and $D_2=(p)$ for some $p,q\in\Gamma$, then $\normalize(f_{D_2-D_1})(x)=j_q(x,p)$.
Now let $D_1=\sum_{i=1}^{d_1} m_{1,i}\cdot(p_{1,i})$ and $D_2=\sum_{i=1}^{d_2} m_{2,i}\cdot(p_{2,i})$ such that $D_1,D_2\in\RDivPlusD(\Gamma)$ (this means $d=\sum_{i=1}^{d_1} m_{1,i}=\sum_{i=1}^{d_2} m_{2,i}$). Then by the linearity of the Laplacian, for an arbitrary $q\in\Gamma$, $$\sum_{i=1}^{d_1}m_{1,i}\cdot j_q(x,p_{1,i})-\sum_{i=1}^{d_2}m_{2,i}\cdot j_q(x,p_{2,i})$$ is an associated function of $D_2-D_1$ .

Define the \emph{distance function} $$\rho(D_1,D_2) := \max(f_{D_2-D_1})-\min(f_{D_2-D_1}) = \max(\normalize(f_{D_2-D_1})).$$ Immediately, we get $\rho(D_1,D_2)=0$ if and only if $D_1=D_2$. Furthermore, note that $$\normalize(f_{D_3-D_1})=\normalize(f_{D_2-D_1}+f_{D_3-D_2}).$$ By the linearity of the Laplacian, we get the triangle inequality $$\rho(D_1,D_3)\leqslant\rho(D_1,D_2)+\rho(D_2,D_3)$$ since $$\normalize(f_{D_2-D_1}+f_{D_3-D_2})\leqslant\normalize(f_{D_2-D_1})+\normalize(f_{D_3-D_2}),$$ while the equalities hold if and only if
$$\Gmin(f_{D_2-D_1})\bigcap\Gmin(f_{D_3-D_2})\neq\emptyset$$ and $$\Gmax(f_{D_2-D_1})\bigcap\Gmax(f_{D_3-D_2})\neq\emptyset.$$
Thus $\rho$ is well-defined as a metric on $\RDivPlusD(\Gamma)$.

Still, we let $D_1,D_2\in\RDivPlusD(\Gamma)$. Let $D_1=D_{1,1}+D_{1,2}$ and $D_2=D_{2,1}+D_{2,2}$. Here we suppose $D_{1,1}$ and $D_{2,1}$ are effective divisors of the same degree $d_1$, and $D_{1,2}$ and $D_{2,2}$ are effective divisors of the same degree $d_2$. By the linearity of the Laplacian, we get $$\normalize(f_{D_2-D_1})=\normalize(f_{D_{2,1}-D_{1,1}}+f_{D_{2,2}-D_{1,2}})$$ and $$\rho(D_1,D_2)\leqslant\rho(D_{1,1},D_{2,1})+\rho(D_{1,2},D_{2,2}),$$ since $$D_2-D_1 = (D_{2,1}-D_{1,1})+(D_{2,2}-D_{1,2}).$$

The \emph{tropical path} (or \emph{t-path}) from $D_1$ to $D_2$ in $\RDivPlusD(\Gamma)$ is a map  $P_{D_2-D_1}:[0,1]\rightarrow \RDivPlusD(\Gamma)$ given by $$P_{D_2-D_1}(t)=\Delta \min(t\cdot\rho(D_1,D_2),\normalize(f_{D_2-D_1}))+D_1.$$ In particular, $P_{D_2-D_1}(0)=D_1$ and $P_{D_2-D_1}(1) = D_2$.

\begin{rmk}
\begin{enumerate}
\item This map is well-defined since $P_{D_2-D_1}(t)$ lies in $\RDivPlusD(\Gamma)$. In other words, there exists a unique t-path from $D_1$ to $D_2$.
\item If we let $D(t) = P_{D_2-D_1}(t)$, then $$\normalize(f_{D(t)-D_1}) = \min(t\cdot\rho(D_1,D_2),\normalize(f_{D_2-D_1})),$$ and $$\normalize(f_{D_2-D(t)}) = \normalize(\max(t\cdot\rho(D_1,D_2),\normalize(f_{D_2-D_1}))).$$
\item $P_{D_1-D_2}$ is continuous.
\end{enumerate}
\end{rmk}

We call $\imag(P_{D_2-D_1})$ the \emph{tropical segment} (or \emph{t-segment}) connecting $D_1$ and $D_2$. Note that $P_{D_2-D_1}(t)=P_{D_1-D_2}(1-t)$ and therefore  $\imag(P_{D_2-D_1})=\imag(P_{D_1-D_2})$. We say $D_1$ and $D_2$ are the \emph{end points} of the t-segment $\imag(P_{D_2-D_1})$.

Given a function $f$ with domain $[\kappa_1,\kappa_2]$ for some $\kappa_1\leqslant\kappa_2$, we say the function $f\diamond s_\alpha$ is a \emph{linear scaling} of $f$ with $\alpha>0$ the scaling factor such that $f\diamond s_\alpha(t)=f(t/\alpha)$, and the function $f\diamond \tau_\beta$ is a \emph{linear translation} of $f$ with $\beta$ the translation factor such that $f\diamond \tau_\beta(t)=f(t-\beta)$. Then it is clear $f\diamond s_\alpha$ has domain $[\alpha\kappa_1,\alpha\kappa_2]$ and $f\diamond \tau_\beta$ has domain $[\kappa_1+\beta,\kappa_2+\beta]$.

$P_{D_2-D_1}$ is actually an isometry after a linear scaling. We give a basic characterization of $P_{D_2-D_1}$ in the following lemma.


\begin{lem} \label{lem:tpath}
For $D_1,D_2\in\RDivPlusD(\Gamma)$, we have the following fundamental properties of the t-path $P_{D_2-D_1}$.
\begin{enumerate}
\item For any $D'_1,D'_2\in\imag(P_{D_2-D_1})$, the t-segment $\imag(P_{D'_2-D'_1})$ is a subset of the t-segment $\imag(P_{D_2-D_1})$.
\item Let $\hat{P}_{D_2-D_1}:[0,\rho(D_1,D_2)]\rightarrow \RDivPlusD(\Gamma)$ be given by $\hat{P}_{D_2-D_1}(t)=P_{D_2-D_1}\diamond s_{\rho(D_1,D_2)}$ if $D_1\neq D_2$ and $\hat{P}_{D_2-D_1}(0)=D_1$ if $D_1=D_2$. Then $\hat{P}_{D_2-D_1}$ is an isometry from $[0,\rho(D_1,D_2)]$ to $\imag(P_{D_2-D_1})$.
\item The t-segment $\imag(P_{D_2-D_1})$ is compact and thus a closed subset of $\RDivPlusD(\Gamma)$.
\end{enumerate}
\end{lem}
\begin{proof}
We may write uniquely $D'_1=P_{D_2-D_1}(t_1)$ and $D'_2=P_{D_2-D_1}(t_2)$ where $t_1,t_2\in[0,1]$. Switching the positions of $D'_1$ and $D'_2$ if necessary, we may assume $t_1\leqslant t_2$.
Then $$\normalize(f_{D'_2-D'_1}) = \normalize(\max(t_1\cdot\rho(D_1,D_2),\min(t_2\cdot\rho(D_1,D_2),\normalize(f_{D_2-D_1})))).$$
Thus we have $\imag(P_{D'_2-D'_1})=P_{D_2-D_1}([t_1,t_2])\subseteq\imag(P_{D_2-D_1})$ (for statement (1)) and $\rho(D'_1,D'_2)=(t_2-t_1)\cdot\rho(D_1,D_2)$ (for statement (2)).

The compactness of $\imag(P_{D_2-D_1})$ follows from the compactness of $[0,1]$ and the continuity of $P_{D_2-D_1}$.
\end{proof}

\begin{cor} \label{cor:tseg_intersect}
The intersection of two t-segments in $\RDivPlusD(\Gamma)$ is again a t-segment in $\RDivPlusD(\Gamma)$.
\end{cor}
\begin{proof}
Let $T_1$ and $T_2$ be two t-segments in $\RDivPlusD(\Gamma)$ with $T$ being their intersection. Then by Lemma~\ref{lem:tpath}~(1), if $T$ contains two divisors $D_1$ and $D_2$, then it must contain the whole t-segment connecting $D_1$ and $D_2$. This actually means that $T$ is either a t-segment itself or a t-segment without one or both of the end points. But $T$ must also be a compact closed subset of $\RDivPlusD(\Gamma)$ by Lemma~\ref{lem:tpath}~(3). Thus $T$ is a t-segment itself.
\end{proof}

\begin{rmk} \label{rmk:Gminmax}
Suppose $D_1\neq D_2$ and we have the t-path $P_{D_2-D_1}$ from $D_1$ to $D_2$ with an associated function $f_{D_2-D_1}$. In particular, we may assume $f_{D_2-D_1}=\normalize(f_{D_2-D_1})$. To simplify notation, we let $D(t) = P_{D_2-D_1}(t)$ and $l=\rho(D_1,D_2)$. Then it is easy to see that
\begin{enumerate}
    \item $\Gmin(f_{D(t)-D_1})=\Gamma$ for $t=0$, and $\Gmin(f_{D(t)-D_1})=\Gmin(f_{D_2-D_1})$ for $t\in(0,1]$;
\item $\Gmax(f_{D(t)-D_1})=f_{D_2-D_1}^{-1}([tl,l])$ for $t\in[0,1]$, and $\Gmax(f_{D(t)-D_1})$ shrinks as $t$ increases; in addition, $\Gmax(f_{D(t)-D_1})$ shrinks continuously as $t$ increase in $(0,s)$ for some $s$ small enough and  $\lim_{t\searrow0}\Gmax(f_{D(t)-D_1})=\overline{(\Gmin(f_{D_2-D_1})^c}$.
\item $\Gmin(f_{D_2-D(t)})=f_{D_2-D_1}^{-1}([0,tl])$ for $t\in[0,1]$, and $\Gmin(f_{D_2-D(t)})$ expands as $t$ increases; in addition, $\Gmin(f_{D_2-D(t)})$ expands continuously as $t$ increase in $(s',1)$ for some $s'$ big enough and $\lim_{t\nearrow1}\Gmin(f_{D_2-D(t)})=\overline{(\Gmax(f_{D_2-D_1})^c}$.
\item $\Gmax(f_{D_2-D(t)})=\Gamma$ for $t=1$, and $\Gmax(f_{D_2-D(t)})=\Gmax(f_{D_2-D_1})$ for $t\in[0,1)$;
\item $\Gmin(f_{D_2-D_1})\bigcap\supp(D_1)\neq\emptyset$ and $\Gmax(f_{D_2-D_1})\bigcap\supp(D_2)\neq\emptyset$; and
\item Let $X=\Gmax(f_{D(t)-D_1})$. Let $X^o$, $X^c$ and $\partial X$ be the interior, complement and boundary of $X$, respectively. Then $D(t)|_{X^o}=D_1|_{X^o}$, $D(t)|_{X^c}=D_2|_{X^c}$ and $D(t)|_{\partial X}\geqslant D_2|_{\partial X}$.
\end{enumerate}
\end{rmk}

\begin{lem} \label{lem:clutching}
Let $D,D_1,D_2\in\RDivPlusD(\Gamma)$. Then the following properties are equivalent.
\begin{enumerate}
\item $D\in\imag(P_{D_2-D_1})$.
\item $\imag(P_{D_2-D_1})=\imag(P_{D_1-D})\bigcup\imag(P_{D_2-D})$.
\item $\Gmin(f_{D_1-D})\bigcup\Gmin(f_{D_2-D})=\Gamma$.
\end{enumerate}
\end{lem}
\begin{proof}
The equivalence of (1) and (2) is straightforward from the definition of the tropical paths.
The equivalence of (2) and (3) follows from the facts that $\Gmin(f_{D_1-D})=\Gmax(f_{D-D_1})$ and $f_{D_2-D}+f_{D-D_1}$ is an associated function of $D_2-D_1$.
\end{proof}

\begin{rmk}
One should be careful that $\rho(D_1,D_2) = \rho(D_1,D)+\rho(D_2,D)$ does not guarantee that $D$ lies in the t-segment connecting $D_1$ and $D_2$.
\end{rmk}

Recall that Corollary~\ref{cor:tseg_intersect} says we will get a t-segment by intersecting two t-segments. The following corollary tells us that if glued properly, the union of two t-segments will also be a t-segment.
\begin{cor}
For $0\leqslant t_1 <t_2\leqslant1$, let $\Lambda:[0,1]\rightarrow \RDivPlusD(\Gamma)$ be a map such that $\Lambda|_{[0,t_2]}\diamond s_{\frac{1}{t_2}}$ is the t-path from $\Lambda(0)$ to $\Lambda(t_2)$ and $\Lambda|_{[t_1,1]}\diamond \tau_{-t_1}\diamond s_{\frac{1}{1-t_1}}$ is the t-path from $\Lambda(t_1)$ to $\Lambda(1)$. Then $\Lambda$ is the t-path from $\Lambda(0)$ to $\Lambda(1)$.
\end{cor}
\begin{proof}
Under the assumptions, we have $\Lambda(t_1)\in\imag(P_{\Lambda(t_2)-\Lambda(0)})=\Lambda([0,t_2])$ and $\Lambda(t_2)\in\imag(P_{\Lambda(1)-\Lambda(t_1)})=\Lambda([t_1,1])$.
Note that a special case is that $\Lambda(t_1)=\Lambda(t_2)$, which implies $\Lambda(0)=\Lambda(1)=\Lambda(t_1)$ since $t_2>t_1$. Now we assume $\Lambda(t_1)\neq\Lambda(t_2)$.
Applying Lemma~\ref{lem:clutching}, we get $$\Gmin(f_{\Lambda(0)-\Lambda(t_1)})\bigcup\Gmin(f_{\Lambda(t_2)-\Lambda(t_1)})=\Gamma.$$ By Remark~\ref{rmk:Gminmax}, we get $$\Gmin(f_{\Lambda(t_2)-\Lambda(t_1)})=\Gmin(f_{\Lambda(1)-\Lambda(t_1)}).$$

Therefore, $$\Gmin(f_{\Lambda(0)-\Lambda(t_1)})\bigcup\Gmin(f_{\Lambda(1)-\Lambda(t_1)})=\Gamma,$$ and it again follows from Lemma~\ref{lem:clutching} that $\Lambda(t_1)\in\imag(P_{\Lambda(1)-\Lambda(0)})$.

Using a similar argument, we get $\Lambda(t_2)\in\imag(P_{\Lambda(1)-\Lambda(0)})$. Thus
$$\imag(P_{\Lambda(1)-\Lambda(0)})=\imag(P_{\Lambda(t_2)-\Lambda(0)})\bigcup\imag(P_{\Lambda(1)-\Lambda(t_1)}=\imag(\Lambda).$$
Note that $$\rho(\Lambda(t_1),\Lambda(t_2))=\frac{t_2-t_1}{t_2}\rho(\Lambda(0),\Lambda(t_2))=\frac{t_2-t_1}{1-t_1}\rho(\Lambda(t_1),\Lambda(1)).$$ Therefore, we must have $\Lambda=P_{\Lambda(1)-\Lambda(0)}$ as claimed.
\end{proof}

If $d$ is an integer and $S_d$ is the symmetric group of degree $d$, then $\DivPlusD(\Gamma)=\Gamma^d/S_d$ set-theoretically. Therefore, other than the metric topology, $\DivPlusD(\Gamma)$ has a topology induced from $\Gamma$ as a $d$-fold symmetric product. The following proposition says that these two topologies on $\DivPlusD(\Gamma)$ are actually the same.

\begin{prop}
On $\DivPlusD(\Gamma)$, the metric topology is the same as the the induced topology as a $d$-fold symmetric product of $\Gamma$.
\end{prop}
\begin{proof}
Denote the first topology by $\mathcal{T}_1$ and the second by $\mathcal{T}_2$. To show $\mathcal{T}_1=\mathcal{T}_2$, it suffices to show that for a divisor $D=\sum_{i=1}^d (q_i)$ with $q_i\in\Gamma$, a sequence $\{D^{(n)}\}_n$ converges to $D$ in $\mathcal{T}_2$ if and only if $\rho(D^{(n)},D)\rightarrow 0$. In addition, we note that to say $D^{(n)}\rightarrow D$ in $\mathcal{T}_2$ is equivalent to say that there exists $d$ sequences of points on $\Gamma$, $\{p_i^{(n)}\}_n$ for $i=1,\dots,d$, such that $D^{(n)}=\sum_{i=1}^d(p_i^{(n)})$ and $p_i^{(n)}\rightarrow q_i$ on $\Gamma$.

Suppose $D^{(n)}\rightarrow D$ in $\mathcal{T}_2$. Since $$\rho(D^{(n)},D)\leqslant\sum_{i=1}^d\rho((p_i^{(n)}),(q_i))=\sum_{i=1}^d r(p_i^{(n)},q_i) \leqslant\sum_{i=1}^d\dist(p_i^{(n)},q_i)$$ where $r(p_i^{(n)},q_i)$ is the effective index between $p_i^{(n)}$ and $q_i$ (see Section~\ref{S:potential}),
we conclude that $D^{(n)}\rightarrow D$ in $\mathcal{T}_1$.

Now suppose $D^{(n)}\rightarrow D$ in $\mathcal{T}_1$ which means $\rho(D^{(n)},D)=\max(\normalize(f_{D^{(n)}-D}))\rightarrow 0$.
Considering the divisors $D$ and $D^{(n)}$, for each point $q_i\in\supp D$, we will associate a point $p_i^{(n)}\supp D^{(n)}$ with an procedure as follows.

Let $M$ be the maximum number of degrees among all the points in $\Gamma$. This means each point $p\in\Gamma$ has at most $M$ adjacent edges. Denote the sum of slopes of $f_{D^{(n)}-D}$ for all outgoing directions from $p\in\Gamma$ by $\chi(p)$. Then $\chi(p)=-(\Delta f_{D^{(n)}-D})(p)=D(p)-D^{(n)}(p)$. Let $V(\Gamma)$ be a vertex set of $\Gamma$.

First, we will determine $p_1^{(n)}$ for $q_1$.

If $q_1\in\supp D^{(n)}$, we let $p_1^{(n)}=q_1$.

Otherwise, we must have $\chi(q_1)\geqslant 1$ and there must be an outgoing direction $\vec{V}_{q_1}$ from $q_1$ with a slope at least $1/M$. Let $w(q_1)\in V(\Gamma)$ be the adjacent vertex of $q_1$ in direction $\vec{V}_{q_1}$. If there exists a point in $\supp D^{(n)}$ that lies in the half-open-half-closed segment $(q_1,w(q_1)]$, then we let $p_1^{(n)}$ be this point. Clearly, $f_{D^{(n)}-D}(q_1)<f_{D^{(n)}-D}(p_1^{(n)})$ and $\dist(p_1^{(n)},q_1)\leqslant M\cdot\rho(D^{(n)},D)$ in this case.

Otherwise, we must have $\chi(w(q_1))\geqslant 0$. Since the slope of the outgoing direction from $w(q_1)$ to $q_1$ is at most $-1/M$, the sum of slopes in the remaining outgoing directions from $w(q_1)$ is at least $1/M$ and there must be an outgoing direction $\vec{V}_{w(q_1)}$ from $w(q_1)$ with a slope at least $1/(M(M-1))$. Let $w^2(q_1)\in V(\Gamma)$ be the adjacent vertex of $w(q_1)$ in direction $\vec{V}_{w(q_1)}$. Following the same procedure, we let $p_1^{(n)}$ be a point contained in both $\supp D^{(n)}$ and $(w(q_1),w^2(q_1)]$ if their intersection is nonempty, and otherwise keep seeking $p_1^{(n)}$ in the next outgoing direction from $w^2(q_1)$ with slope at least $1/(M(M-1)^2)$.

The procedure must terminate in finitely many steps since we only have finitely many elements in $V(\Gamma)$. Let $N=|V(\Gamma)|$. We conclude that
we can find $p_1^{(n)}$ within $N$ steps and $\dist(p_1^{(n)},q_1)\leqslant C_1\cdot\rho(D^{(n)},D)$ where $C_1=M(M-1)^N$.

Next we will determine $p_i^{(n)}$ one by one inductively. Suppose for $i=2,\ldots,d'$ ($d'<d$), we have determined $p_i^{(n)}$ and known that $\dist(p_i^{(n)},q_i)\leqslant C_i\cdot\rho(D^{(n)},D)$ where $C_i$'s are constants. We let $D^{(n)}_{d'}=D^{(n)}-\sum_{i=1}^{d'}(p_i^{(n)})$ and
$D_{d'}=D-\sum_{i=1}^{d'}(q_i)$. Then
\begin{align*}
\rho(D^{(n)}_{d'},D_{d'})&\leqslant\rho(D^{(n)},D)+\sum_{i=1}^{d'}r(p_i^{(n)},q_i) \\
&\leqslant\rho(D^{(n)},D)+\sum_{i=1}^{d'}\dist(p_i^{(n)},q_i) \\
&=(1+\sum_{i=1}^{d'}C_i)\rho(D^{(n)},D).
\end{align*}
Following exactly the same procedure we used to seek $p_1^{(n)}$, we can find $p_{d'+1}^{(n)}\in\supp D^{(n)}_{d'}$ such that
$$\dist(p_{d'+1}^{(n)},q_{d'+1})\leqslant C_1\cdot\rho(D^{(n)}_{d'},D_{d'})=C_{d'+1}\cdot\rho(D^{(n)},D)$$ where $C_{d'+1}=C_1(1+\sum_{i=1}^{d'}C_i)$.

In this way, for each $D^{(n)}$, we can find $p_i^{(n)}$ such that $D^{(n)}=\sum_{i=1}^d(p_i^{(n)})$ and $\dist(p_i^{(n)},q_i)$ is bounded by $C_i\cdot\rho(D^{(n)},D)$. This means $D^{(n)}\rightarrow D$ in $\mathcal{T}_1$ implies $D^{(n)}\rightarrow D$ in $\mathcal{T}_2$.
\end{proof}

\begin{lem} \label{lem:rescale}
The scaling map $\phi:\RDiv_+^{d'}(\Gamma)\rightarrow\RDiv_+^{d}(\Gamma)$ given by $\phi(D)=\frac{d}{d'}D$ is a homeomorphism. Moreover, $$\rho(\phi(D_1),\phi(D_2))=\frac{d}{d'}\rho(D_1,D_2)$$ for $D_1,D_2\in\RDiv_+^{d}(\Gamma)$.
\end{lem}
\begin{proof}
It follows directly from the linearity of the Laplacian.
\end{proof}


\section{Tropical convex sets: a generalization of complete linear systems} \label{S:tconvset}
\begin{dfn}
A set $T\subseteq\RDivPlusD(\Gamma)$ is \emph{tropically convex} (\emph{t-convex}) or equivalently \emph{t-path-connected} of degree $d$ if for every $D_1,D_2\in T$, the whole t-segment $imag(P_{D_2-D_1})$ connecting $D_1$ and $D_2$ is contained in $T$. 
\end{dfn}

Note that the intersection of an arbitrary collection of tropically convex sets of the same degree is tropically convex. Thus we define the \emph{tropical convex hull} generated by $S\subseteq\RDivPlusD(\Gamma)$, denoted by $\tconv(S)$, as the intersection of all tropically convex sets in $\RDivPlusD(\Gamma)$ containing $S$, and we say $S$ is a \emph{generating set} of $\tconv(S)$. If, in addition, $x\notin \tconv(S\setminus\{x\})$ for every $x\in S$, then we say $S$ is \emph{tropical convex (t-convex) independent}.
We say a tropical convex hull is \emph{finitely generated} if it can be generated by a finite set. In particular, we abuse notation here to write $\tconv(D_1,\ldots,D_n, S_1,\ldots,S_m)$ as a simplification of $\tconv(\{D_1,\ldots,D_n\}\bigcup S_1\ldots \bigcup S_m)$ when it is clear that $D_1,\ldots,D_n$ are divisors in $\RDivPlusD(\Gamma)$ and $S_1\ldots \bigcup S_m$ are subsets of $\RDivPlusD(\Gamma)$. In particular, by Lemma~\ref{lem:tpath}~(1), it is easy to verify that $\tconv(D_1,D_2)=\imag(P_{D_2-D_1})$, and we use them both interchangeably to represent the t-segment connecting $D_1$ and $D_2$.

If $d$ is an integer and $D_1,D_2\in\DivPlusD(\Gamma)$, we say $D_1$ is linearly equivalent to $D_2$ (denoted $D_1\sim D_2$) if $f_{D_2-D_1}$ is rational, i.e., piecewise-linear with integral slopes. This is equivalent to say $\tconv(D_1,D_2)\subseteq \DivPlusD(\Gamma)$. The complete linear system $|D|$ associated to $D\in\DivPlusD(\Gamma)$ is the set of effective divisors linearly equivalent to $D$.

We have the following facts:
\begin{enumerate}
  \item All complete linear systems $|D|$ are t-path-connected.
  \item $\DivPlusD(\Gamma)$ is not t-path-connected in general, and the nonempty complete linear systems of degree $d$ are the t-path-connected components in $\DivPlusD(\Gamma)$.
  \item $\RDivPlusD(\Gamma)$ is t-path-connected, but not finitely generated. When $d$ is an integer, we have in general $\RDivPlusD(\Gamma)\supsetneq\tconv(\DivPlusD(\Gamma))$. 
\end{enumerate}

\begin{lem} \label{lem:complete_finite}
Every complete linear system is finitely generated.
\end{lem}
\begin{proof}
A complete linear system $|D|$ can always be generated by the extremals (we only have finitely many of them) in $|D|$.
\end{proof}

\begin{rmk}
The extremals of complete linear systems are introduced in \cite{HMY12}. (They actually define extremals in $L(D)$ instead of $|D|$.) We will generalize this notion to all tropical convex sets in Section~\ref{S:RD_tconv}.
\end{rmk}

Now let us consider tropical convex sets in general. Theorem~\ref{thm:contractible} and Theorem~\ref{thm:construction} state some fundamental properties of tropical convex sets. In particular, as it is well-known that conventional convex subsets of Euclidean spaces are contractible, Theorem~\ref{thm:contractible} says this is also true for all tropical convex sets. Theorem~\ref{thm:construction} tells us how to generate a tropical convex set from its subsets and provides a compactness criterion. Then we may deduce an important conclusion immediately that finitely generated tropical convex hulls are always compact (Corollary~\ref{cor:finite_gen}). To prove these theorems, we need to employ a machinery based on \emph{general reduced divisors} which will be introduced in the next section, and we will finish the proofs in Section~\ref{S:RD_tconv}.

\begin{thm} \label{thm:contractible}
Tropical convex sets are contractible.
\end{thm}

\begin{thm}\label{thm:construction}
Let $T,T'\subseteq\RDivPlusD(\Gamma)$ be tropically convex set. Then we have $\tconv(T,T')=\bigcup_{D\in T,D'\in T'}\tconv(D,D')$. If $T$ and $T'$ are compact in addition, then $\tconv(T,T')$ is compact.
\end{thm}

\begin{cor} \label{cor:finite_gen}
Every finitely generated tropical convex hull is compact.
\end{cor}
\begin{proof}
It follows immediately from Lemma~\ref{lem:tpath}~(3) and an induction on Theorem~\ref{thm:construction}.
\end{proof}

\begin{rmk} \label{rmk:complete_compact}
The complete linear systems are finitely generated (Lemma~\ref{lem:complete_finite}) and thus compact in our metric topology (Corollary~\ref{cor:finite_gen}).
\end{rmk}

\section{General reduced divisors} \label{S:GRD}
\subsection{$\Sfunc$-functions}
Let the $\Sfunc$-function $\Sfunc:\RDiv^0(\Gamma)\rightarrow \mR_+$ be given by $\Sfunc(D_2-D_1)=\int_\Gamma(f_{D_2-D_1}-\min(f_{D_2-D_1}))=\int_\Gamma\normalize(f_{D_2-D_1})$, where $D_1$ and $D_2$ are effective $\mR$-divisors of the same degree. In addition, for $d\geqslant0$, we define the $\Sfunc$-function restricted to degree $d$ as $\Sfunc^d:\RDivPlusD(\Gamma)\times\RDivPlusD(\Gamma)\rightarrow \mR_+$ given by $\Sfunc^d(D_1,D_2)=\Sfunc(D_2-D_1)$. Unlike the distance function, we have $\Sfunc(D_2-D_1)\neq \Sfunc(D_1-D_2)$ in general. It is straightforward to verify that (1) $\Sfunc(D_1-D_2)+\Sfunc(D_2-D_1)=\rho(D_1,D_2)l_\mathrm{tot}$ where $l_\mathrm{tot}$ is the total length of $\Gamma$, and (2) $\Sfunc(D_1-D_2)=0$ if and only if $\rho(D_1,D_2)=0$. Fixing $D_1$ or $D_2$, we get the functions $\Sfunc_{\star-D_1}:\RDivPlusD(\Gamma)\rightarrow \mR_+$ given by $\Sfunc_{\star-D_1}(D)=\Sfunc(D-D_1)$ and $\Sfunc_{D_2-\star}:\RDivPlusD(\Gamma)\rightarrow \mR_+$ given by $\Sfunc_{D_2-\star}(D)=\Sfunc(D_2-D)$, respectively.

\begin{rmk} \label{rmk:bfunction}
For $D\in\DivPlusD(\Gamma)$ and $q\in\Gamma$, the $b$-function $b_q(D)$ Baker and Shokrieh introduced in \cite{BF12} is essentially a special case of the $\Sfunc$-function in the following sense:
$$b_q(D) = \Sfunc(D-d\cdot(q)).$$
\end{rmk}

\begin{lem} \label{lem:Sfunc}
\begin{enumerate}
\item For $D_1,D_2,D_3\in\RDivPlusD(\Gamma)$, we have the triangle inequality $$\Sfunc(D_3-D_1)\leqslant\Sfunc(D_3-D_2)+\Sfunc(D_2-D_1).$$ The equality holds if and only if $$\Gmin(f_{D_3-D_2})\bigcap\Gmin(f_{D_2-D_1})\neq\emptyset$$ if and only if
    $$\Gmin(f_{D_3-D_1})=\Gmin(f_{D_3-D_2})\bigcap\Gmin(f_{D_2-D_1}).$$
\item For $D_1,D_2,D_3\in\RDivPlusD(\Gamma)$, $\rho(D_1,D_3)=\rho(D_1,D_2)+\rho(D_2,D_3)$ if and only if \\
      $$\Sfunc(D_3-D_1)=\Sfunc(D_3-D_2)+\Sfunc(D_2-D_1)$$ and
      $$\Sfunc(D_1-D_3)=\Sfunc(D_1-D_2)+\Sfunc(D_2-D_3).$$
\item The functions $\Sfunc^d$, $\Sfunc_{\star-D}$ and $\Sfunc_{D-\star}$ are continuous.
\end{enumerate}
\end{lem}
\begin{proof}
For the triangle inequality, we let $f_{D_3-D_2}$ and $f_{D_2-D_1}$ be associated to $D_3-D_2$ and $D_2-D_1$ respectively, and assume $f_{D_3-D_2}=\normalize(f_{D_3-D_2})$ and $f_{D_2-D_1}=\normalize(f_{D_2-D_1})$. Let $f_{D_3-D_1}=f_{D_3-D_2}+f_{D_2-D_1}$, which is associated to $D_3-D_1$. Note that $$\min(f_{D_3-D_1})\geqslant\min(f_{D_3-D_2})+\min(f_{D_2-D_1})=0,$$ while the equality holds if and only if $$\Gmin(f_{D_3-D_2})\bigcap\Gmin(f_{D_2-D_1})\neq\emptyset$$ if and only if
    $$\Gmin(f_{D_3-D_1})=\Gmin(f_{D_3-D_2})\bigcap\Gmin(f_{D_2-D_1}).$$
Thus
\begin{align*}
\Sfunc(D_3-D_1)&=\int_\Gamma(f_{D_3-D_1}-\min(f_{D_3-D_1}))\\
&\leqslant\int_\Gamma f_{D_3-D_1}\\
&=\int_\Gamma f_{D_3-D_2} + \int_\Gamma f_{D_2-D_1}\\
&=\Sfunc(D_3-D_2)+\Sfunc(D_2-D_1),
\end{align*}
with the equality holds under the same conditions.

For (2), $\rho(D_1,D_3)=\rho(D_1,D_2)+\rho(D_2,D_3)$ if and only if
$$\Gmin(f_{D_2-D_1})\bigcap\Gmin(f_{D_3-D_2})\neq\emptyset$$ and $$\Gmax(f_{D_2-D_1})\bigcap\Gmax(f_{D_3-D_2})\neq\emptyset.$$ Note that $\Gmax(f_{D_2-D_1})=\Gmin(f_{D_1-D_2})$ and
$\Gmax(f_{D_3-D_2})=\Gmin(f_{D_2-D_3})$, and hence (2) follows from (1).

For (3), it suffices to show $\Sfunc(D'_2-D'_1)\rightarrow\Sfunc(D_2-D_1)$ as $D'_1\rightarrow D_1$ and $D'_2\rightarrow D_2$. Actually, if $l_\mathrm{tot}$ is the total length of $\Gamma$, we have
\begin{align*}
\Sfunc(D'_2-D'_1)-\Sfunc(D_2-D_1)&=\Sfunc((D'_2-D_2)+(D_2-D_1)+(D_1-D'_1))-\Sfunc(D_2-D_1) \\
&\leqslant\Sfunc(D'_2-D_2)+\Sfunc(D_1-D'_1) \\
&\leqslant(\rho(D_2,D'_2)+\rho(D_1,D'_1))l_\mathrm{tot}
\end{align*}
and
\begin{align*}
\Sfunc(D_2-D_1)-\Sfunc(D'_2-D'_1)&=\Sfunc((D_2-D'_2)+(D'_2-D'_1)+(D'_1-D_1))-\Sfunc(D'_2-D'_1) \\
&\leqslant\Sfunc(D_2-D'_2)+\Sfunc(D'_1-D_1) \\
&\leqslant(\rho(D_2,D'_2)+\rho(D_1,D'_1))l_\mathrm{tot}.
\end{align*}
\end{proof}

\subsection{General reduced divisors}
\begin{thm} \label{thm:main}
Let $T\subseteq\RDivPlusD(\Gamma)$ be tropically convex and compact. For every $E\in\RDivPlusD(\Gamma)$, there exists a unique $\mR$-divisor $T_E\in T$, which minimizes $\Sfunc_{\star-E}|_T$.
\end{thm}
According to Lemma~\ref{lem:Sfunc}~(3), $\Sfunc_{\star-E}$ is a continuous function. Since $T$ is compact, $\Sfunc_{\star-E}|_T$ can reach its minimal value. Hence, it only remains to show that the minimum can only be reached at a single divisor in $T$. We will finish our proof of Theorem~\ref{thm:main} in Remark~\ref{rmk:working} after proving some useful facts in Proposition~\ref{prop:working}. Provided this theorem, we are now ready to bring up a central notion of this paper.

\begin{dfn} \label{dfn:GRD:same_deg}
Under the hypotheses of Theorem~\ref{thm:main}, we say the divisor $T_E$ is the \emph{(general) reduced divisor} in $T$ with respect to $E$ (or the $E$-reduced divisor in $T$).
\end{dfn}

\begin{rmk}
For $D\in\DivPlusD(\Gamma)$ and $q\in\Gamma$, Baker and Shokrieh \cite{BF12} showed that a conventional reduced divisor $D_q$ is the unique divisor in the complete linear system $|D|$ such that the $b$-function $b_q(D)$ is minimized. Note that $|D|$ is compact (Remark~\ref{rmk:complete_compact}) and we may express the $b$-function by an equivalent $\Sfunc$-function (Remark~\ref{rmk:bfunction}). Hence if we let $T=|D|$ and $E=d\cdot(q)$, the conventional reduced divisors fit well in our new setting by the identity $D_q=|D|_{d\cdot(q)}$.
\end{rmk}

\begin{rmk}
Throughout this paper, when we mention reduced divisors, we mean general reduced divisors unless otherwise stated.
\end{rmk}

\begin{prop} \label{prop:working}
Let $E,D_1,D_2\in\RDivPlusD(\Gamma)$ and $D_1\neq D_2$. Let $P_{D_2-D_1}$ be the t-path from $D_1$ to $D_2$. Let $D(t) = P_{D_2-D_1}(t)$ for $t\in[0,1]$. Consider the functions $g_\rho(t) = \rho(E,D(t))$ and $g_{\Sfunc}(t)=\Sfunc(D(t)-E)$ for $t\in[0,1]$. Then exactly one of the following two cases occur:
\begin{enumerate}
\item $\Gmin(f_{D_1-E})\bigcap\Gmin(f_{D_2-D_1})\neq\emptyset$. In this case, $g_\rho(t)$ is increasing and $g_{\Sfunc}(t)$ is strictly increasing for $t\in[0,1]$. And precisely, for $t\in(0,1]$, we have
    \begin{align*}
    \Gmin(f_{D(t)-E})&=\Gmin(f_{D(t)-D_1})\bigcap\Gmin(f_{D_1-E})\\
    =\Gmin(f_{D_2-E})&=\Gmin(f_{D_2-D_1})\bigcap\Gmin(f_{D_1-E})
    \end{align*}
    and $g_{\Sfunc}(t)=\Sfunc(D(t)-D_1)+\Sfunc(D_1-E)$.
\item $\Gmin(f_{D_1-E})\bigcap\Gmin(f_{D_2-D_1})=\emptyset$. In this case, at $t=0$, $g_\rho(t)$ is decreasing and $g_{\Sfunc}(t)$ is strictly decreasing.
\end{enumerate}
\end{prop}
\begin{rmk}
We say a function $f(t)$ is increasing (resp. decreasing, strictly increasing, strictly decreasing, or locally constant) at $t_0$ if there exists $\delta>0$ such that $g(t)$ is increasing (resp. decreasing, strictly increasing, strictly decreasing, or constant) on $[t_0,t_0+\delta]$. Note that we adopt the usual definition of increasing (resp. decreasing) functions here, which actually means non-decreasing (resp. non-increasing).
\end{rmk}
\begin{proof}
Let $l=\rho(D_1,D_2)$. For simplicity of notations, we assume $$\min(f_{D_1-E})=\min(f_{D_2-D_1})=\min(f_{D(t)-D_1})=0$$ from now on. It then follows $f_{D(t)-D_1}=\min(tl,f_{D_2-D_1})$. In addition, we let $f_{D_2-E}=f_{D_2-D_1}+f_{D_1-E}$, which is associated to $D_2-E$, and  $f_{D(t)-E}=f_{D(t)-D_1}+f_{D_1-E}$, which is associated to $D(t)-E$.

If $$\Gmin(f_{D_1-E})\bigcap\Gmin(f_{D_2-D_1})\neq\emptyset,$$ then we have
$$\Gmin(f_{D_2-E})=\Gmin(f_{D_2-D_1})\bigcap\Gmin(f_{D_1-E})$$ and
$$\min(f_{D_2-E})=\min(f_{D_2-D_1}+f_{D_1-E})=\min(f_{D_2-D_1})+\min(f_{D_1-E})=0.$$
By Remark~\ref{rmk:Gminmax}~(1), $\Gmin(f_{D(t)-D_1})=\Gmin(f_{D_2-D_1})$ for $t\in(0,1]$.
Therefore,
$$\Gmin(f_{D(t)-E})=\Gmin(f_{D(t)-D_1})\bigcap\Gmin(f_{D_1-E})=\Gmin(f_{D_2-D_1})\bigcap\Gmin(f_{D_1-E})\neq\emptyset$$
and
$$\min(f_{D(t)-E})=\min(f_{D(t)-D_1}+f_{D_1-E})=\min(f_{D(t)-D_1})+\min(f_{D_1-E})=0$$ for $t\in[0,1]$.
On the other hand, $\max(f_{D(t)-E})$ is an increasing function since $\max(f_{D(t)-E})=\max(f_{D(t)-D_1}+f_{D_1-E})$ and the value of $f_{D(t)-D_1}(v)$ at any point $v\in\Gamma$ is an increasing function with respect to $t$.
Therefore, $g_\rho(t)$ is also an increasing function since $g_\rho(t)=\max(f_{D(t)-E})-\min(f_{D(t)-E})=\max(f_{D(t)-E})$.
Moreover, it follows from Lemma~\ref{lem:Sfunc} that
\begin{align*}
g_{\Sfunc}(t)&=\Sfunc(D(t)-E)\\
&=\Sfunc(D(t)-D_1)+\Sfunc(D_1-E).
\end{align*}
Therefore $g_{\Sfunc}(t)$ is strictly increasing for $t\in[0,1]$ since $\Sfunc(D(t)-D_1)$ is strictly increasing.

Now consider the case $\Gmin(f_{D_2-D_1})\bigcap\Gmin(f_{D_1-E})=\emptyset$.
Note that both $f_{D_2-D_1}^{-1}([0,\delta])$ and $f_{D_1-E}^{-1}([0,\delta])$ are closed subsets of $\Gamma$ with finitely many connected components, and for a small enough positive $\delta_0$, both $f_{D_2-D_1}^{-1}([0,\delta])$ and $f_{D_1-E}^{-1}([0,\delta])$ expand continuously as $\delta$ increases in $[0,\delta_0]$. In particular,
we have $$\lim_{\delta\searrow0}f_{D_2-D_1}^{-1}([0,\delta])=\Gmin(f_{D_2-D_1})$$ and $$\lim_{\delta\searrow0}f_{D_1-E}^{-1}([0,\delta])=\Gmin(f_{D_1-E}).$$ Hence we may even choose $\delta_0$ such that $$f_{D_2-D_1}^{-1}([0,\delta])\bigcap f_{D_1-E}^{-1}([0,\delta]) =\emptyset$$ for all $\delta\in[0,\delta_0]$.
Then for $t\in[0,\delta_0/l]$, we have
\begin{itemize}
  \item $f_{D(t)-D_1}=tl$, $f_{D_1-E}(v)=0$ and $f_{D(t)-E}(v)=tl$ if $v\in\Gmin(f_{D_1-E})$;
  \item $f_{D(t)-D_1}\geqslant0$, $f_{D_1-E}(v)\geqslant tl$ and $f_{D(t)-E}(v)\geqslant tl$ if $v\in f_{D_2-D_1}^{-1}([0,tl])$; and
  \item $f_{D(t)-D_1}=tl$, $f_{D_1-E}(v)>0$ and $f_{D(t)-E}(v)>tl$ if $v\in(f_{D_2-D_1}^{-1}([0,tl])\bigcup\Gmin(f_{D_1-E}))^c$.
\end{itemize}
Therefore, we conclude $\min(f_{D(t)-E})=tl$ and $\Gmin(f_{D_1-E})\subseteq\Gmin(f_{D(t)-E})$ for $t\in[0,\delta_0/l]$.
Let $f_{D_1-D(t)}=\rho(D_1,D(t))-f_{D(t)-D_1}$, and we have $\min(f_{D_1-D(t)})=0$ and the value of $f_{D_1-D(t)}(v)$ at any point $v\in\Gamma$ is an increasing function with respect to $t$.
Then for $t\in[0,\delta_0/l]$,
\begin{align*}
\normalize(f_{D(t)-E})&=f_{D(t)-E}-tl \\
&=f_{D_1-E}+f_{D(t)-D_1}-\rho(D_1,D(t)) \\
&=f_{D_1-E}-f_{D_1-D(t)}.
\end{align*}
Note that $g_\rho(t)=\max(\normalize(f_{D(t)-E}))$, which means $g_\rho(t)$ is decreasing for $t\in[0,\delta_0/l]$. Thus $g_\rho(t)$ is decreasing at $t=0$.
Moreover,
\begin{align*}
g_{\Sfunc}(t)&=\Sfunc(D(t)-E)\\
&=\int_\Gamma\normalize(f_{D(t)-E}) \\
&=\int_\Gamma(f_{D_1-E}-f_{D_1-D(t)}) \\
&=\Sfunc(D_1-E)-\Sfunc(D_1-D(t)),
\end{align*}
for $t\in[0,\delta_0/l]$. This means $g_{\Sfunc}(t)$ is strictly decreasing for $t\in[0,\delta_0/l]$ since $\Sfunc(D_1-D(t))$ is strictly increasing. Thus $g_{\Sfunc}(t)$ is strictly decreasing at $t=0$.
\end{proof}

\begin{rmk} \label{rmk:working}
We observe some easy facts following from Proposition~\ref{prop:working}.
\begin{enumerate}
\item $g_\rho(t)$ can be locally constant, while $g_{\Sfunc}(t)$ cannot.
\item If $g_\rho(t)$ is \emph{strictly increasing} at $t=0$, then $g_\rho(t)$ is \emph{increasing} on $[0,1]$. If $g_{\Sfunc}(t)$ is \emph{strictly increasing} at $t=0$, then $g_{\Sfunc}(t)$ is \emph{strictly increasing} on $[0,1]$.
\item Recall that we've assumed $D_1\neq D_2$. If $g_\rho(0)=g_\rho(1)=\kappa_\rho$, then $g_{\Sfunc}(t)$ is decreasing at $t=0$ (locally constant is possible) and $g_\rho(t)\leqslant\kappa_\rho$ for $t\in(0,1)$. If $g_{\Sfunc}(0)=g_{\Sfunc}(1)=\kappa_{\Sfunc}$, then $g_{\Sfunc}(t)$ is strictly decreasing at $t=0$ and $g_{\Sfunc}(t)<\kappa_{\Sfunc}$ for $t\in(0,1)$.
\item We can finish \textbf{the proof of Theorem~\ref{thm:main}} now. If there exist divisors $D_1$ and $D_2$ in $T$, both minimizing $\Sfunc_{\star-E}|_{D\in T}$, then we must have $D_1=D_2$ by (3).
\item By applying Proposition~\ref{prop:working} to the t-paths from $D_1$ to $D_2$ and from $D_1$ to $D_2$ respectively, we see that $$\Gmin(f_{D_2-D_1})\bigcap\Gmin(f_{D_1-E})\neq\emptyset$$ implies $$\Gmin(f_{D_1-D_2})\bigcap\Gmin(f_{D_2-E})=\emptyset$$ (still under the assumption $D_1\neq D_2$).
\end{enumerate}
\end{rmk}


Proposition~\ref{prop:working} can actually provide us with criterions of reduced divisors from different aspects, as summarized in the following corollary.
\begin{cor}[\textbf{Criterions for general reduced divisors}] \label{cor:reduced}
Let $T\subseteq\RDivPlusD(\Gamma)$ be tropically convex and compact.
Let $E\in\RDivPlusD(\Gamma)$ and $D_0\in T$. The following properties are equivalent.
\begin{enumerate}
\item $D_0$ is the $E$-reduced divisor of $T$.
\item For every $D\in T$ and $t\in[0,1]$, the function $\Sfunc(P_{D-D_0}(t)-E)$ is strictly increasing.
\item For every $D\in T$ and $t\in[0,1]$, the function $\Sfunc(P_{D-D_0}(t)-E)$ is strictly increasing at $t=0$. (Equivalently, we say $\Sfunc_{\star-E}$ is strictly increasing at $D_0$ along all possible firing directions.)
\item For every $D\in T$, $$\Gmin(f_{D-D_0})\bigcap\Gmin(f_{D_0-E})\neq\emptyset.$$
\item For every $D\in T$, $$\Gmin(f_{D-E})=\Gmin(f_{D-D_0})\bigcap\Gmin(f_{D_0-E}).$$
\item For every $D\in T$, $$\Sfunc(D-E)=\Sfunc(D-D_0)+\Sfunc(D_0-E).$$
\item For every $D\in T$ and $D\neq D_0$, $$\Gmin(f_{D_0-D})\bigcap\Gmin(f_{D-E})=\emptyset.$$
\end{enumerate}
\end{cor}
\begin{proof}
All the criterions easily follows from Proposition~\ref{prop:working}.
\end{proof}

\subsection{Some properties of general reduced divisors}
Unless otherwise stated, we let $T\subseteq\RDivPlusD(\Gamma)$ be tropically convex and compact in the following discussions.
\begin{lem} \label{lem:identity}
If $E\in T$, then $T_E=E$.
\end{lem}

\begin{lem} \label{lem:red_on_subconv}
Let $T'$ be a compact tropical convex subset of $T$. For $E\in\RDivPlusD(\Gamma)$, if $T_E\in T'$, then $T'_E=T_E$.
\end{lem}

The easy facts as stated in the above two lemmas can be verified using any criterion of reduced divisors in Corollary~\ref{cor:reduced}, and we skip the detailed proofs.

\begin{lem} \label{lem:red_fiber}
Let $E'\in\tconv(E,T_E)$. Then $T_{E'}=T_{E}$.
\end{lem}
\begin{proof}
By Corollary~\ref{cor:reduced}, we have $$\Gmin(f_{D-T_E})\bigcap\Gmin(f_{T_E-E})\neq\emptyset$$ for every $D\in T$.
Taking $D_1=E$ and $D_2=T_E$ in Remark~\ref{rmk:Gminmax}~(3), we have $\Gmin(f_{T_E-E'})\supseteq\Gmin(f_{T_E-E})$, which means $$\Gmin(f_{D-T_E})\bigcap\Gmin(f_{T_E-E'})\neq\emptyset$$ for every $D\in T$. Using Corollary~\ref{cor:reduced} again, we see that $T_E$ is also $E'$-reduced in $T$.
\end{proof}

\begin{lem}
For $D_0,E,E'\in\RDivPlusD(\Gamma)$, suppose $D_0\in T$ and $E'\in\tconv(E,D_0)$. Then $E'\in\tconv(E,T_{E'})$.
\end{lem}
\begin{proof}
By Corollary~\ref{cor:reduced}, $\Gmin(f_{D_0-E'})\subseteq\Gmin(f_{T_{E'}-E'})$. By Lemma~\ref{lem:clutching},
$$\Gmin(f_{D_0-E'})\bigcup\Gmin(f_{E-E'})=\Gamma.$$ Thus $$\Gmin(f_{T_{E'}-E'})\bigcup\Gmin(f_{E-E'})=\Gamma.$$ Again, by Lemma~\ref{lem:clutching}, we have $E'\in\tconv(E,T_{E'})$.
\end{proof}

\begin{lem} \label{lem:dist_decrease}
Let $E_1,E_2\in\RDivPlusD(\Gamma)$. Then $\rho(T_{E_1},T_{E_2})\leqslant\rho(E_1,E_2)$. The equality holds if and only if $$\Sfunc(T_{E_2}-E_1)=\Sfunc(T_{E_2}-E_2)+\Sfunc(E_2-E_1)$$ and
$$\Sfunc(T_{E_1}-E_2)=\Sfunc(T_{E_1}-E_1)+\Sfunc(E_1-E_2).$$
\end{lem}
\begin{proof}
Let $l_\mathrm{tot}$ be the total length of $\Gamma$.
By Corollary~\ref{cor:reduced}, we have
$$\Sfunc(T_{E_2}-T_{E_1})=\Sfunc(T_{E_2}-E_1)-\Sfunc(T_{E_1}-E_1)$$ and
$$\Sfunc(T_{E_1}-T_{E_2})=\Sfunc(T_{E_1}-E_2)-\Sfunc(T_{E_2}-E_2).$$
By Lemma\ref{lem:Sfunc}, we have
$$\Sfunc(T_{E_2}-E_1)-\Sfunc(T_{E_2}-E_2)\leqslant\Sfunc(E_2-E_1)$$ and
$$\Sfunc(T_{E_1}-E_2)-\Sfunc(T_{E_1}-E_1)\leqslant\Sfunc(E_1-E_2).$$
Therefore,
\begin{align*}
&\rho(T_{E_1},T_{E_2})l_\mathrm{tot} \\
&=\Sfunc(T_{E_2}-T_{E_1})+\Sfunc(T_{E_1}-T_{E_2}) \\
&=(\Sfunc(T_{E_2}-E_1)-\Sfunc(T_{E_1}-E_1))+(\Sfunc(T_{E_1}-E_2)-\Sfunc(T_{E_2}-E_2)) \\
&=(\Sfunc(T_{E_2}-E_1)-\Sfunc(T_{E_2}-E_2))+(\Sfunc(T_{E_1}-E_2)-\Sfunc(T_{E_1}-E_1)) \\
&\leqslant\Sfunc(E_2-E_1)+\Sfunc(E_1-E_2) \\
&=\rho(E_1,E_2)l_\mathrm{tot}.
\end{align*}
\end{proof}

\begin{cor} \label{cor:dist_eq}
Let $E_1,E_2\in\RDivPlusD(\Gamma)$. If $\rho(E_1,E_2)=\rho(T_{E_1},T_{E_2})$, then for each $E\in\RDivPlusD(\Gamma)$ such that $\rho(E_1,E)+\rho(E_2,E)=\rho(E_1,E_2)$, we have $\rho(E_1,E)=\rho(T_{E_1},T_E)$ and $\rho(E_2,E)=\rho(T_{E_2},T_E)$.
\end{cor}
\begin{proof}
By Lemma\ref{lem:dist_decrease}, we get $\rho(T_{E_1},T_E)\leqslant\rho(E_1,E)$ and $\rho(T_{E_2},T_E)\leqslant\rho(E_2,E)$. Thus
$$\rho(E_1,E_2) = \rho(T_{E_1},T_{E_2})\leqslant \rho(T_{E_1},T_E)+\rho(T_{E_2},T_E)\leqslant\rho(E_1,E) + \rho(E_2,E) = \rho(E_1,E_2),$$
which implies $\rho(E_1,E)=\rho(T_{E_1},T_E)$ and $\rho(E_2,E)=\rho(T_{E_2},T_E)$.
\end{proof}

\begin{rmk}
Each divisor $E\in\tconv(E_1,E_2)$ satisfies the condition $\rho(E_1,E)+\rho(E_2,E)=\rho(E_1,E_2)$ in Corollary~\ref{cor:dist_eq}. Therefore, we must have $\rho(E_1,E)=\rho(T_{E_1},T_E)$ and $\rho(E_2,E)=\rho(T_{E_2},T_E)$. However, we should note that the set $\{T_E:E\in\tconv(E_1,E_2)\}$ is not necessarily a tropical convex set.
\end{rmk}

\begin{lem} \label{lem:RD_subspace}
Let $E\in\RDivPlusD(\Gamma)$ and $T'$ be a compact tropical convex subset of $T$. Then $T'_E=T'_{T_E}$.
\end{lem}
\begin{proof}
To prove $T'_E=T'_{T_E}$, it suffices to show that $$\Sfunc(D'-E)=\Sfunc(D'-T'_{T_E})+\Sfunc(T'_{T_E}-E)$$ for every $D'\in T'$ by Corollary~\ref{cor:reduced}.

Actually, applying Corollary~\ref{cor:reduced} to $T$ with respect to $E$, we get
$$\Sfunc(D-E)=\Sfunc(D-T_E)+\Sfunc(T_E-E)$$ for every $D\in T$, and in particular
$$\Sfunc(T'_{T_E}-E)=\Sfunc(T'_{T_E}-T_E)+\Sfunc(D_1-E).$$

Applying Corollary~\ref{cor:reduced} to $T'$ with respect to $T_E$, we get $$\Sfunc(D'-T_E)=\Sfunc(D'-T'_{T_E})+\Sfunc(T'_{T_E}-T_E)$$ for every $D'\in T'$.

Therefore,
\begin{align*}
& \Sfunc(D'-E)=\Sfunc(D'-T_E)+\Sfunc(T_E-E) \\
&= \Sfunc(D'-T'_{T_E})+\Sfunc(T'_{T_E}-T_E)+\Sfunc(T_E-E) \\
&= \Sfunc(D'-T'_{T_E})+\Sfunc(T'_{T_E}-E)
\end{align*}
for every $D'\in T'$, and $T'_{T_E}$ is exactly the $E$-reduced divisor in $T'$ as claimed.

%
\end{proof}

Let $E\in\RDivPlusD(\Gamma)$ and $r_{\min}=\inf_{D\in T}\rho(E,D)$ (knowing $T$ is compact, actually we have $r_{\min}=\min_{D\in T}\rho(E,D)$). The following proposition shows that sublevel sets of the distance function $\rho_E:=\rho(E,\star)$ and the $\Sfunc$-function $\Sfunc_{\star-E}$ on $T$ are all tropically convex.
For $r,s\in\mR_+$, we let $L_{\leqslant r}^{T}(\rho_E)=\{D\in T| \rho(E,D)\leqslant r\}$, $L_{=r}^{T}(\rho_E)=\{D\in T| \rho(E,D)=r\}$, $L_{\leqslant s}^{T}(\Sfunc_{\star-E})=\{D\in T| \Sfunc_{\star-E}(D)\leqslant s\}$, and $L_{=s}^{T}(\Sfunc_{\star-E})=\{D\in T| \Sfunc_{\star-E}(D)=s\}$. In particular, we also denote the the level set $L_{=r_{\min}}^{T}(\rho_E)$ of $\rho_E$ at the minimum distance by $L_{\min}^{T}(\rho_E)$.

\begin{prop} \label{prop:levelset_pt}
Under the above hypotheses and notations, we have
\begin{enumerate}
\item The $E$-reduced divisor $T_E$ lies in $L_{\min}^{T}(\rho_E)$.
\item $L_{\min}^{T}(\rho_E)$, $L_{\leqslant r}^{T}(\rho_E)$, $L_{=r}^{T}(\rho_E)$, $L_{\leqslant s}^{T}(\Sfunc_{\star-E})$ and $L_{=s}^{T}(\Sfunc_{\star-E})$ are all compact subsets of $T$.
\item $L_{\min}^{T}(\rho_E)$, $L_{\leqslant r}^{T}(\rho_E)$ and $L_{\leqslant s}^{T}(\Sfunc_{\star-E})$ are tropically convex with the compactness assumption of $T$ removed.
\end{enumerate}
\end{prop}
\begin{proof}
Let $D$ be any divisor in $T$. By Corollary~\ref{cor:reduced}, we have
$$\Gmin(f_{D-E})=\Gmin(f_{D-T_E})\bigcap\Gmin(f_{T_E-E})\neq\emptyset.$$ Therefore,
\begin{align*}
\rho(E,D)&=\max(\normalize(f_{D-E}))=\max(\normalize(f_{D-T_E})+\normalize(f_{T_E-E}))\\
&\geqslant\max(\normalize(f_{T_E-E}))=\rho(E,T_E),
\end{align*}
which implies $\rho(E,T_E)=r_{\min}$ and thus $T_E\in L_{\min}^{T}(\rho_E)$.

For (2), the compactness of $L_{\min}^{T}(\rho_E)$, $L_{\leqslant r}^{T}(\rho_E)$, $L_{=r}^{T}(\rho_E)$, $L_{\leqslant s}^{T}(\Sfunc_{\star-E})$ and $L_{=s}^{T}(\Sfunc_{\star-E})$ follows from the compactness of $T$ and the continuity of the distance function and the $\Sfunc$-function.

Now let us show $L_{\leqslant r}^{T}(\rho_E)$ and $L_{\leqslant s}^{T}(\Sfunc_{\star-E})$ are tropically convex. In the following arguments, we do not require $T$ to be compact. The tropical convexity of $L_{\min}^{T}(\rho_E)$ will follow from the tropical convexity of $L_{\leqslant r}^{T}(\rho_E)$ by setting $r=r_{\min}$. By Proposition~\ref{prop:working} and Remark~\ref{rmk:working}, if $D_1,D_2\in L_{\leqslant r}^{T}(\rho_E)$, then $$\rho(E,D)\leqslant\max\{\rho(E,D_1),\rho(E,D_2)\}\leqslant r$$ for all $D$ in $\tconv(D_1,D_2)$ and thus $\tconv(D_1,D_2)\subseteq L_{\leqslant r}^{T}(\rho_E)$. Respectively, if $D_1,D_2\in L_{\leqslant s}^{T}(\Sfunc_{\star-E})$, then $$\Sfunc_{\star-E}(D)<\max\{\Sfunc_{\star-E}(D_1),\Sfunc_{\star-E}(D_2)\}\leqslant s$$ for all $D$ in the interior of $\tconv(D_1,D_2)$ and thus $\tconv(D_1,D_2)\subseteq L_{\leqslant s}^{T}(\Sfunc_{\star-E})$. Therefore, both $L_{\leqslant r}^{T}(\rho_E)$ and $L_{\leqslant s}^{T}(\Sfunc_{\star-E})$ are tropically convex.
\end{proof}

\section{Reduced divisors in tropical segments} \label{S:RD_tseg}
As t-segments are tropically convex and compact (Lemma~\ref{lem:tpath}), the reduced divisors are well-defined for t-segments. In this section, we study the properties of reduced divisors in t-segments, and the results will be employed intensively in the next section where we give proofs to some prestated theorems.

\subsection{Basic properties}
\begin{lem} \label{lem:RD_on_tsegment}
For $E,D_1,D_2\in\RDivPlusD(\Gamma)$, let $D_0$ be the $E$-reduced divisor in $\tconv(D_1,D_2)$. Then we have
$$\Gmin(f_{D_0-E})=\Gmin(f_{D_1-E})\bigcup\Gmin(f_{D_2-E}),$$ and for all $D\in\tconv(D_1,D_2)$, $$\Gmin(f_{D-E})\subseteq\Gmin(f_{D_1-E})\bigcup\Gmin(f_{D_2-E}).$$
\end{lem}
\begin{proof}
Applying Corollary~\ref{cor:reduced} to $\tconv(D_1,D_2)$ with respect to $E$ and knowing that $D_0$ is the corresponding reduced divisor, we have
$$\Gmin(f_{D_1-E})=\Gmin(f_{D_1-D_0})\bigcap\Gmin(f_{D_0-E}),$$
$$\Gmin(f_{D_2-E})=\Gmin(f_{D_2-D_0})\bigcap\Gmin(f_{D_0-E}),$$
and $$\Gmin(f_{D-E})=\Gmin(f_{D-D_0})\bigcap\Gmin(f_{D_0-E}).$$
Moreover, we have $$\Gmin(f_{D_1-D_0})\bigcap\Gmin(f_{D_2-D_0})=\Gamma$$ by Lemma~\ref{lem:clutching}. Therefore,
$$\Gmin(f_{D-E})\subseteq\Gmin(f_{D_0-E})=\Gmin(f_{D_1-E})\bigcup\Gmin(f_{D_2-E}).$$
\end{proof}

\begin{lem} \label{lem:RD_sub_tsegment}
 Let $E,D_1,D_2\in\RDivPlusD(\Gamma)$ and $D'_1,D'_2\in\tconv(D_1,D_2)$. Suppose $D'_1\in\tconv(D_1,D'_2)$. Let $D_0$ be the $E$-reduced divisor in $\tconv(D_1,D_2)$ and $D'_0$ be the $E$-reduced divisor in $\tconv(D'_1,D'_2)$. Then
\begin{enumerate}
\item $D'_0=D_0$ if and only if $D_0\in\tconv(D'_1,D'_2)$;
\item $D'_0=D'_1$ if and only if $D_0\in\tconv(D_1,D'_1)$;
\item $D'_0=D'_2$ if and only if $D_0\in\tconv(D'_2,D_2)$.
\end{enumerate}
\end{lem}

\begin{proof}
 This is an immediate consequence of the fact that the functions $\Sfunc(P_{D_1-D_0}(t)-E)$ and $\Sfunc(P_{D_2-D_0}(t)-E)$ are both strictly increasing (Corollary~\ref{cor:reduced}).
\end{proof}

\begin{lem} \label{lem:RD_2tsegs}
For $D_1,D_2,D_3\in\RDivPlusD(\Gamma)$, we have
\begin{enumerate}
\item $D_2$ is the $D_1$-reduced divisor in $\tconv(D_2,D_3)$ if and only if $$\Sfunc(D_3-D_1)=\Sfunc(D_3-D_2)+\Sfunc(D_2-D_1).$$
\item $D_2$ is simultaneously the $D_1$-reduced divisor in $\tconv(D_2,D_3)$ and the $D_3$-reduced divisor in $\tconv(D_1,D_2)$ if and only if
    $$\rho(D_1,D_3)=\rho(D_1,D_2)+\rho(D_2,D_3).$$
\end{enumerate}
\end{lem}
\begin{proof}
(1) follows easily from Lemma~\ref{lem:Sfunc}~(1), Proposition~\ref{prop:working} and the criterions for reduced divisors (Corollary~\ref{cor:reduced}).

Recall that by Lemma~\ref{lem:Sfunc}~(2), we have
$\rho(D_1,D_3)=\rho(D_1,D_2)+\rho(D_2,D_3)$ if and only if \\
      $$\Sfunc(D_3-D_1)=\Sfunc(D_3-D_2)+\Sfunc(D_2-D_1)$$ and
      $$\Sfunc(D_1-D_3)=\Sfunc(D_1-D_2)+\Sfunc(D_2-D_3).$$
      Then (2) follows from (1).
\end{proof}

\begin{rmk}
By Lemma~\ref{lem:RD_2tsegs}, the for the sufficient and necessary conditions for equality in Lemma~\ref{lem:dist_decrease} can be equivalently stated as $E_2$ is the $E_1$-reduced divisor in $\tconv(E_2,T_{E_2})$ and $E_1$ is the $E_2$-reduced divisor in $\tconv(E_1,T_{E_1})$.
\end{rmk}

\subsection{Tropical triangles}
Roughly, we may call the tropical convex hull generated by three divisors in $\RDivPlusD(\Gamma)$ a tropical triangle. We will show that tropical triangles are made of tropical segments.
\begin{prop} \label{prop:TExtension}
Let $D_0,D_1,D_2\in\RDivPlusD(\Gamma)$ (see Figure~\ref{Fig:TExtension}), $D_3\in\tconv(D_0,D_1)$ and $D_4\in\tconv(D_0,D_2)$. Then we have we have the following properties.
\begin{enumerate}
\item For every $D_5\in\tconv(D_3,D_4)$, there exists $D'_5\in\tconv(D_1,D_2)$ such that $D_5\in\tconv(D_0,D'_5)$. In particular, we can let $D'_5$ be the $D_5$-reduced divisor in $\tconv(D_1,D_2)$.
\item Conversely, for every $D'_5\in\tconv(D_1,D_2)$, there exists $D_5\in\tconv(D_3,D_4)$ such that $D_5\in\tconv(D_0,D'_5)$. (In other words, $\tconv(D_3,D_4)\bigcap\tconv(D_0,D'_5)\neq\emptyset$.) More precisely, assuming $D'_3$ is the $D_3$-reduced divisor in $\tconv(D_1,D_2)$ and $D'_4$ is the $D_4$-reduced divisor in $\tconv(D_1,D_2)$, we have
    \begin{itemize}
        \item if $D'_5\in\tconv(D'_3,D'_4)$, then $D_5$ can be chosen such that $D'_5$ be the $D_5$-reduced divisor in $\tconv(D_1,D_2)$;
        \item if $D'_5\in\tconv(D_1,D'_3)$, then $D_5$ can be chosen to be $D_3$; and
        \item if $D'_5\in\tconv(D_2,D'_4)$, then $D_5$ can be chosen to be $D_4$.
    \end{itemize}
\end{enumerate}
\end{prop}
\begin{figure}[tbp]
\centering
\includegraphics[width=.4\textwidth]{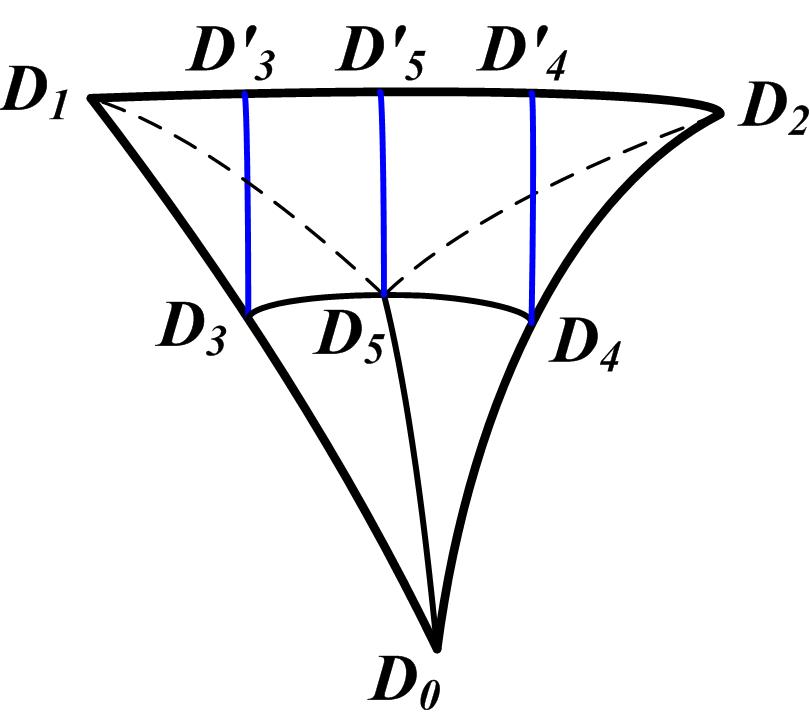}
\caption{}\label{Fig:TExtension}
\end{figure}
\begin{proof}
For (1), we suppose $D'_5$ is the $D_5$-reduced divisor in $\tconv(D_1,D_2)$, and claim that $D_5\in\tconv(D_0,D'_5)$. By Lemma~\ref{lem:RD_on_tsegment}, we have $$\Gmin(f_{D'_5-D_5})=\Gmin(f_{D_1-D_5})\bigcup\Gmin(f_{D_2-D_5}).$$

Applying Lemma~\ref{lem:RD_on_tsegment} again, we have
$$\Gmin(f_{D_3-D_5})\subseteq\Gmin(f_{D_0-D_5})\bigcup\Gmin(f_{D_1-D_5})$$
and $$\Gmin(f_{D_4-D_5})\subseteq\Gmin(f_{D_0-D_5})\bigcup\Gmin(f_{D_2-D_5}).$$
Note that $\Gmin(f_{D_3-D_5})\bigcup\Gmin(f_{D_4-D_5})=\Gamma$ by Lemma~\ref{lem:clutching}. Therefore,
\begin{align*}
&\Gmin(f_{D_0-D_5})\bigcup\Gmin(f_{D'_5-D_5}) \\
&=\Gmin(f_{D_0-D_5})\bigcup(\Gmin(f_{D_1-D_5})\bigcup\Gmin(f_{D_2-D_5}))\\
&=(\Gmin(f_{D_0-D_5})\bigcup\Gmin(f_{D_1-D_5}))\bigcup(\Gmin(f_{D_0-D_5})\bigcup\Gmin(f_{D_2-D_5}))\\
&\supseteq\Gmin(f_{D_3-D_5})\bigcup\Gmin(f_{D_4-D_5})=\Gamma,
\end{align*}
which means $D_5\in\tconv(D_0,D'_5)$ by Lemma~\ref{lem:clutching}.

For (2), we need to use a fact in Section~\ref{S:Canproj} that reduced-divisor maps (Definition~\ref{def:RDM}) are continuous (Lemma~\ref{lem:RDM_continuous}). Then it follows that if $D'_5\in\tconv(D'_3,D'_4)$, then there exists $D_5\in\tconv(D_3,D_4)$ such that $D'_5$ be the $D_5$-reduced divisor in $\tconv(D_1,D_2)$. By (1), this also means that $D_5\in\tconv(D_0,D'_5)$ as expected.

If $D'_5\in\tconv(D_1,D'_3)$, then by Proposition~\ref{prop:working} and Corollary~\ref{cor:reduced},
$$\Gmin(f_{D'_5-D_3})=\Gmin(f_{D'_5-D'_3})\bigcap\Gmin(f_{D'_3-D_3})$$ and
$$\Gmin(f_{D_1-D_3})=\Gmin(f_{D_1-D'_3})\bigcap\Gmin(f_{D'_3-D_3}),$$ which imply $\Gmin(f_{D_1-D_3})\subseteq\Gmin(f_{D'_5-D_3})$. (Actually, if in addition $D'_5\neq D'_3$, then \\ $\Gmin(f_{D_1-D_3})=\Gmin(f_{D'_5-D_3})$.) By Lemma~\ref{lem:clutching}, since $D_3\in\tconv(D_0,D_1)$ which implies
$$\Gmin(f_{D_0-D_3})\bigcup\Gmin(f_{D_1-D_3})=\Gamma,$$ we have
$$\Gmin(f_{D_0-D_3})\bigcup\Gmin(f_{D'_5-D_3})=\Gamma$$ which implies $D_3\in\tconv(D_0,D'_5)$.

If $D'_5\in\tconv(D_2,D'_4)$, a similar argument can show that $D_4\in\tconv(D_0,D'_5)$.
\end{proof}

\begin{rmk}
In our proof of Proposition~\ref{prop:TExtension}~(2), in the case that $D'_5\in\tconv(D'_3,D'_4)$ and $D'_5$ is the $D_5$-reduced divisor in $\tconv(D_1,D_2)$, we do not need an additional assumption that $D'_3,D'_5,D'_4$ lie in $\tconv(D_1,D_2)$ in the same order as $D_3,D_5,D_4$ lie in $\tconv(D_3,D_4)$ as illustrated in Figure~\ref{Fig:TExtension}. But this is actually true, i.e., we must have $D'_3\in\tconv(D_1,D'_4)$ (or equivalently $D'_4\in\tconv(D_2,D'_3)$) and $D'_5\in\tconv(D'_3,D'_4)$. Here is why. First we show that $D'_3\in\tconv(D_1,D'_4)$. If $D'_3\notin\tconv(D_1,D'_4)$, then $D_3\neq D_4$. Referring to our proof of Proposition~\ref{prop:TExtension}~(2), we see that $D_3,D_4\in\tconv(D_0,D'_3)$ and $D_3,D_4\in\tconv(D_0,D'_4)$.
Let us draw contradictions from all possible cases. Recall that by Lemma~\ref{lem:red_fiber}, given a compact tropical convex set $T$, a divisor $E$ of the same degree and $T_E$ the corresponding $E$-reduced divisor in $T$, all the divisors on $\tconv(E,T_E)$ share the same reduced divisor in $T$.
\begin{itemize}
\item $D_4\in\tconv(D_3,D'_3)$: It implies $D'_4=D'_3$, a contradiction.
\item $D_3\in\tconv(D_4,D'_4)$: It implies $D'_3=D'_4$, a contradiction.
\item $D_4\in\tconv(D_0,D_3)$: It goes back to the case $D_3\in\tconv(D_4,D'_4)$. (To see this, you may want to use Lemma~\ref{lem:clutching} and refer to our proof of Proposition~\ref{prop:TExtension}~(2).)
\item $D_3\in\tconv(D_0,D_4)$: It goes back to the case $D_4\in\tconv(D_3,D'_3)$.
\end{itemize}
Thus we get $D'_3\in\tconv(D_1,D'_4)$ as claimed. Now suppose there exists $D_5\in\tconv(D_3,D_4)$ such that $D'_5\notin\tconv(D'_3,D'_4)$. Actually we may suppose $D'_5\in\tconv(D_1,D'_3)\setminus\{D'_3\}$ and $D'_3\in\tconv(D'_4,D'_5)$. Then by the continuity of reduced-divisor maps, there must exist $D_6\in\tconv(D_4,D_5)$ such that $D'_3$ is also the $D_6$-reduced divisor in $\tconv(D_1,D_2)$. Then following from Proposition~\ref{prop:TExtension}~(1), both $D_3$ and $D_6$ lie in $\tconv(D_0,D'_3)$. Since $D_5\in\tconv(D_3,D_6)$, we get $D'_5=D'_3$ no matter $D_6\in\tconv(D_3,D'_3)$ or $D_3\in\tconv(D_6,D'_3)$ by Lemma~\ref{lem:red_fiber}, which is a contradiction.
\end{rmk}

\begin{rmk}
There are several aspects of Proposition~\ref{prop:TExtension}. First, as in (1), if we choose arbitrarily a divisor (e.g. $D_3$) in $\tconv(D_0,D_1)$, a divisor (e.g. $D_4$) in $\tconv(D_0,D_2)$, and then arbitrarily a divisor (e.g. $D_5$) in $\tconv(D_3,D_4)$, we may add a t-segment $\tconv(D_5,D'_5)$ with $D'_5\in\tconv(D_1,D_2)$ to the t-segment $\tconv(D_0,D_5)$ while the result of such an extension is exactly $\tconv(D_0,D'_5)$. With one step further, we can derive Corollary~\ref{cor:tri_tconvex}, which is a special case of Theorem~\ref{thm:construction}. Second, the $D_5$-reduced divisor in $\tconv(D_1,D_2)$ (as we've done throughout the proof) is a desired choice for $D'_5$. On the other hand, in some cases, we can choose $D'_5$ which is not necessarily $D_5$-reduced. Third, as in (2), it says that $\tconv(D_3,D_4)$ and $\tconv(D_0,D'_5)$ must intersect. But the intersection might not be just a single point.
Example~\ref{ex:TExtension} gives a concrete demonstration of these phenomena.
\end{rmk}

\begin{cor} \label{cor:tri_tconvex}
For $D_0,D_1,D_2\in\RDivPlusD(\Gamma)$, choose arbitrarily $D'_1$ in $\tconv(D_0,D_1)$ and $D'_2$ in $\tconv(D_0,D_2)$. Then we have
$$\tconv(D'_1,D'_2)\subseteq\bigcup_{D\in \tconv(D_1,D_2)}\tconv(D_0,D)$$ and
$$\tconv(D_0,D_1,D_2)=\bigcup_{D\in \tconv(D_1,D_2)}\tconv(D_0,D).$$
\end{cor}
\begin{proof}
By Proposition~\ref{prop:TExtension}, we see immediately
$$\tconv(D'_1,D'_2)\subseteq\bigcup_{D\in \tconv(D_1,D_2)}\tconv(D_0,D).$$
Then $\bigcup_{D\in \tconv(D_1,D_2)}\tconv(D_0,D)$ is tropically convex by definition, and must be the minimal to contain $D_0$, $D_1$ and $D_2$.
Thus $$\tconv(D_0,D_1,D_2)=\bigcup_{D\in \tconv(D_1,D_2)}\tconv(D_0,D).$$
\end{proof}

%

\subsection{Useful length inequalities}
\begin{prop} \label{prop:dist_ineq}
For $D^0_1,D^0_2,D_1,D_2\in\RDivPlusD(\Gamma)$, let $E_1\in\tconv(D^0_1,D_1)$ and $E_2\in\tconv(D^0_2,D_2)$. Let $D'_1$ be the $E_2$-reduced divisor in $\tconv(D_0,D_1)$ and $D'_2$ the $E_1$-reduced divisor in $\tconv(D_0,D_2)$. If $D'_1\in\tconv(D^0_1,E_1)$ and $D'_2\in\tconv(D^0_2,E_2)$, then $\rho(E_1,E_2)\leqslant\rho(D''_1,D''_2)$ for all $D''_1\in\tconv(E_1,D_1)$ and $D''_2\in\tconv(E_2,D_2)$.
\end{prop}
\begin{figure}[tbp]
\centering
\includegraphics[width=.4\textwidth]{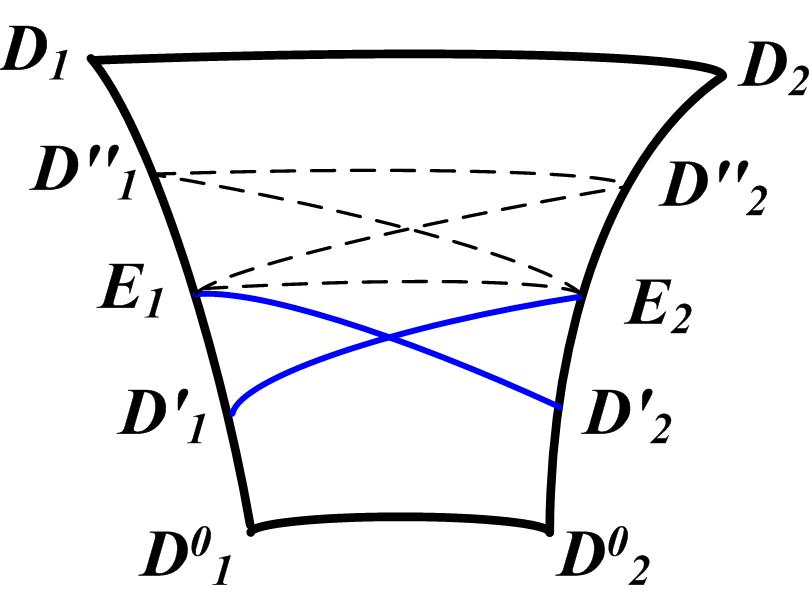}
\caption{}\label{Fig:dist_ineq}
\end{figure}
\begin{proof}
Let $l_\mathrm{tot}$ be the total length of $\Gamma$.
Under the assumptions and applying Lemma~\ref{lem:RD_sub_tsegment}, $D'_1$ must also be the $E_2$-reduced divisor in both $\tconv(D^0_1,E_1)$ and $\tconv(D^0_1,D''_1)$,  and $D'_2$ must also $E_1$-reduced divisor in both $\tconv(D^0_2,E_2)$ $\tconv(D^0_2,D''_2)$. Therefore, applying Corollary~\ref{cor:reduced}, we get the following equalities.
\begin{align*}
&\Sfunc(D''_1-E_2) \\
&=\Sfunc(D'_1-E_2)+\Sfunc(D''_1-D'_1) \\
&=\Sfunc(D'_1-E_2)+\Sfunc(D''_1-E_1)+\Sfunc(E_1-D'_1) \\
&=\Sfunc(D''_1-E_1)+\Sfunc(E_1-E_2),
\end{align*}
and analogously $$\Sfunc(D''_2-E_1) =\Sfunc(D''_2-E_2)+\Sfunc(E_2-E_1).$$
Therefore,
\begin{align*}
\rho(E_1,E_2)l_\mathrm{tot}&=\Sfunc(E_1-E_2)+\Sfunc(E_2-E_1) \\
&=(\Sfunc(D''_1-E_2)-\Sfunc(D''_1-E_1))+(\Sfunc(D''_2-E_1)-\Sfunc(D''_2-E_2)) \\
&=(\Sfunc(D''_1-E_2)-\Sfunc(D''_2-E_2))+(\Sfunc(D''_2-E_1)-\Sfunc(D''_1-E_1)) \\
&\leqslant \Sfunc(D''_1-D''_2)+\Sfunc(D''_2-D''_1) = \rho(D''_1,D''_2)l_\mathrm{tot}.
\end{align*}
The last inequality follows from the triangle inequality for $\Sfunc$-functions(Lemma~\ref{lem:Sfunc}).
\end{proof}

The following corollaries of Proposition~\ref{prop:dist_ineq} are two special cases convenient for applications.
\begin{cor} \label{cor:red_dist_ineq}
Let $D^0_1,D^0_2,D_1,D_2,E_1,E_2$ be under the same hypotheses as in Proposition~\ref{prop:dist_ineq}. If $E_1$ is the $E_2$-reduced divisor in $\tconv(D^0_1,D_1)$, then $$\rho(E_1,E_2)\leqslant\max(\rho(D_1,D_2),\rho(D^0_1,D^0_2)).$$ In particular, if in addition $D^0_1=D^0_2$, then $\rho(E_1,E_2)\leqslant\rho(D''_1,D''_2)$ for all $D''_1\in\tconv(E_1,D_1)$ and $D''_2\in\tconv(E_2,D_2)$.
\end{cor}
\begin{proof}
Let $D'_1$ be the $E_2$-reduced divisor in $\tconv(D^0_1,D_1)$ and $D'_2$ the $E_1$-reduced divisor in $\tconv(D^0_2,D_2)$. Then $D'_1$ is exactly $E_1$ which means $D'_1\in\tconv(D_0,E_1)$ automatically. Thus by Proposition~\ref{prop:dist_ineq}, if $D'_2\in\tconv(D^0_2,E_2)$, then $\rho(E_1,E_2)\leqslant\rho(D_1,D_2)$, and if
$D'_2\in\tconv(D_2,E_2)$, then $\rho(E_1,E_2)\leqslant\rho(D^0_1,D^0_2)$. In both cases, $\rho(E_1,E_2)\leqslant\max(\rho(D_1,D_2),\rho(D^0_1,D^0_2))$.

Recall that by Lemma\ref{lem:dist_decrease}, the distance between reduced divisors is at most the distance between the original divisors. Thus if in addition $D^0_1=D^0_2=D_0$, then $\rho(D_0,D'_2)\leqslant\rho(D_0,E_1)\leqslant\rho(D_0,E_2)$, which implies $D'_2\in\tconv(D_0,E_2)$. It follows from Proposition~\ref{prop:dist_ineq} that $\rho(E_1,E_2)\leqslant\rho(D''_1,D''_2)$.
\end{proof}

\begin{cor} \label{cor:ball_dist_ineq}
Let $D^0_1,D^0_2,D_1,D_2,E_1,E_2$ be under the same hypotheses as in Proposition~\ref{prop:dist_ineq} and suppose $D^0_1=D^0_2=D_0$. If $\rho(D_0,E_1)=\rho(D_0,E_2)$, then $\rho(E_1,E_2)\leqslant\rho(D''_1,D''_2)$ for all $D''_1\in\tconv(E_1,D_1)$ and $D''_2\in\tconv(E_2,D_2)$.
\end{cor}
\begin{proof}
By Lemma\ref{lem:dist_decrease} and get $\rho(D_0,D'_1)\leqslant\rho(D_0,E_2)$ and $\rho(D_0,D'_2)\leqslant\rho(D_0,E_1)$. Since $\rho(D_0,E_1)=\rho(D_0,E_2)$, we get $\rho(D_0,D'_1)\leqslant\rho(D_0,E_1)$ and $\rho(D_0,D'_2)\leqslant\rho(D_0,E_2)$. Thus we have $D'_1\in\tconv(D_0,E_1)$ and $D'_2\in\tconv(D_0,E_2)$, and it follows from Proposition~\ref{prop:dist_ineq} that $\rho(E_1,E_2)\leqslant\rho(D''_1,D''_2)$.
\end{proof}

\section{A revisit of the general properties of tropical convex sets}  \label{S:RD_tconv}
\subsection{Proofs of Theorem~\ref{thm:contractible} and Theorem~\ref{thm:construction}}

\begin{proof} [\textbf{Proof of Theorem~\ref{thm:contractible}}]
Let $T\subseteq\RDivPlusD(\Gamma)$ be tropically convex.
To show $T$ is contractible, it suffices to find a continuous function $h:[0,1]\times T \rightarrow T$ such that for some $D_0\in T$ and all $D\in T$, $h(0,D) = D$ and $h(1,D) = D_0$. Indeed, we can define the contraction map $h$ as follows. Choose $D_0$ arbitrarily from $T$ and let $\kappa=\sup_{D'\in T}\rho(D_0,D')$. For any $D\in T$, we let
$h(t,D)=D$ if $t\in[0,1-\frac{\rho(D_0,D)}{\kappa})$, and $h(t,D)=P_{D_0-D}(\frac{\kappa}{\rho(D_0,D)}(t-1)+1)$ if $t\in[1-\frac{\rho(D_0,D)}{\kappa},1]$. More explicitly, the contraction happens in the following way: for any $t\in[0,1]$, if $\rho(D_0,D)<\kappa (1-t)$, then $h(t,D)=D$, and otherwise, $h(t,D)$ lies on the t-segment $\tconv(D_0,D)$ with distance $\kappa (1-t)$ to $D_0$. Then it is clear that $h(0,D) = D$ and $h(1,D) = D_0$. Therefore the only remaining fact to verify is the continuity of $h$. In other words, we need to show that $h(t_n,D_n)\rightarrow h(t,D)$ whenever $t_n\rightarrow t$ and $D_n\rightarrow D$ (we let $n>0$ for $D_n$ to avoid confusion with $D_0$).  Note that $$\rho(h(t_n,D_n),h(t,D))\leqslant\rho(h(t_n,D_n),h(t,D_n))+\rho(h(t,D_n),h(t,D)),$$
and $$\rho(h(t_n,D_n),h(t,D_n))\leqslant\rho(D_0,D_n)|t_n-t|\leqslant\kappa|t_n-t|.$$
Therefore, to show the continuity of $h$, it suffices to show $\rho(h(t,D_n),h(t,D))\leqslant\rho(D_n,D)$.

Case (1): $\rho(D_0,D_n)<\kappa (1-t)$ and $\rho(D_0,D)<\kappa (1-t)$. In this case, $h(t,D_n)=D_n$ and $h(t,D)=D$.

Case (2): $\rho(D_0,D_n)<\kappa (1-t)$ and $\rho(D_0,D)\geqslant\kappa (1-t)$. In this case, $h(t,D_n)=D_n$ and $h(t,D)\in\tconv(D_0,D)$ with distance $\kappa (1-t)$ to $D_0$. Let $D'\in\tconv(D_0,D)$ be the $D_n$-reduced divisor in $\tconv(D_0,D)$. Then by Lemma~\ref{lem:dist_decrease}, $$\rho(D_0,D')\leqslant\rho(D_0,D_n)<\kappa (1-t)=\rho(D_0,h(t,D)).$$ This means $D'\in\tconv(D_0,h(t,D))$, and by Proposition~\ref{prop:dist_ineq}, $$\rho(h(t,D_n),h(t,D))=\rho(D_n,h(t,D))\leqslant\rho(D_n,D).$$

Case (3): $\rho(D_0,D_n)\geqslant\kappa (1-t)$ and $\rho(D_0,D)<\kappa (1-t)$. In this case, $h(t,D)=D$ and $h(t,D_n)\in\tconv(D_0,D_n)$ with distance $\kappa (1-t)$ to $D_0$. Let $D'_n\in\tconv(D_0,D_n)$ be the $D$-reduced divisor in $\tconv(D_0,D_n)$. Using an analogous argument as in case (2), we see that $\rho(h(t,D_n),h(t,D))\leqslant\rho(D_n,D)$.

Case (4): $\rho(D_0,D_n)\geqslant\kappa (1-t)$ and $\rho(D_0,D)\geqslant\kappa (1-t)$. In this case, $h(t,D_n)\in\tconv(D_0,D_n)$ and $h(t,D)\in\tconv(D_0,D)$, both with distance $\kappa (1-t)$ to $D_0$. Therefore, by Corollary~\ref{cor:ball_dist_ineq}, we have $\rho(h(t,D_n),h(t,D))\leqslant\rho(D_n,D)$.
\end{proof}

\begin{rmk}
The contraction map $h$ constructed in the above proof deforms the whole $T$ to a point $D_0\in T$. In particular, one can notice that at each $t\in[0,1]$, the set $h(t,T)$ is actually the sublevel set $L_{\leqslant r}^{T}(\rho_{D_0})$ of the distance function $\rho_{D_0}$ to $D_0$ where $r=\kappa (1-t)$. Therefore, $h(t,T)$ is tropically convex by Proposition~\ref{prop:levelset_pt}~(3).
\end{rmk}

\begin{proof} [\textbf{Proof of Theorem~\ref{thm:construction}}]
Denote $\bigcup_{D\in T,D'\in T'}\tconv(D,D')$ by $\tilde{T}$. Then clearly $\tilde{T}\subseteq\tconv(T,D)$. We claim that $\tilde{T}$ is tropically convex, which will imply $\tilde{T}=\tconv(T,D)$.

Choose arbitrarily $E_1$ and $E_2$ from $\tilde{T}$. Then there exist $D_1,D_2\in T$ and $D'_1,D'_2\in T'$ such that $E_1\in\tconv(D_1,D'_1)$ and $E_2\in\tconv(D_2,D'_2)$. Since $T$ and $T'$ are tropically convex, we have $\tconv(D_1,D_2)\subseteq T$ and $\tconv(D'_1,D'_2)\subseteq T'$.
For every $E\in\tconv(E_1,E_2)$, let $D\in\tconv(D_1,D_2)$ be the $E$-reduced divisor in $\tconv(D_1,D_2)$ and $D'\in\tconv(D'_1,D'_2)$ be the $E$-reduced divisor in $\tconv(D'_1,D'_2)$. To show $\tilde{T}$ is tropically convex, it suffices to show that $E\in\tconv(D,D')$.

By Lemma~\ref{lem:RD_on_tsegment}, we have
$$\Gmin(f_{D-E})=\Gmin(f_{D_1-E})\bigcup\Gmin(f_{D_2-E}),$$
$$\Gmin(f_{D'-E})=\Gmin(f_{D'_1-E})\bigcup\Gmin(f_{D'_2-E}),$$
$$\Gmin(f_{E_1-E})\subseteq\Gmin(f_{D_1-E})\bigcup\Gmin(f_{D'_1-E}),$$
and $$\Gmin(f_{E_2-E})\subseteq\Gmin(f_{D_2-E})\bigcup\Gmin(f_{D'_2-E}).$$
Note that since $E\in\tconv(E_1,E_2)$, we have $\Gmin(f_{E_1-E})\bigcup\Gmin(f_{E_2-E})=\Gamma$ by Lemma~\ref{lem:clutching}. Therefore,
\begin{align*}
&\Gmin(f_{D-E})\bigcup\Gmin(f_{D'-E}) \\
&=(\Gmin(f_{D_1-E})\bigcup\Gmin(f_{D_2-E}))\bigcup(\Gmin(f_{D'_1-E})\bigcup\Gmin(f_{D'_2-E}))\\
&=(\Gmin(f_{D_1-E})\bigcup\Gmin(f_{D'_1-E}))\bigcup(\Gmin(f_{D_2-E})\bigcup\Gmin(f_{D'_2-E}))\\
&\supseteq\Gmin(f_{E_1-E})\bigcup\Gmin(f_{E_2-E})=\Gamma,
\end{align*}
which means $E\in\tconv(D,D')$ by Lemma~\ref{lem:clutching}.

Recall that a metric space is compact if and only if it is complete and totally bounded. Now let us show that if in addition $T$ and $T'$ are complete and totally bounded, then $\tilde{T}$ is also complete and totally bounded.

First, we show that $\tilde{T}$ is complete. Let $E_1,E_2,\ldots$ be a Cauchy sequence in $\tilde{T}$, i.e., $\rho(E_m,E_n)\rightarrow 0$ as $m,n\rightarrow\infty$. We claim that there exists $E_0\in\tilde{T}$ such that $\rho(E_n,E_0)\rightarrow 0$ as $n\rightarrow\infty$, which implies the completeness of $\tilde{T}$. Since $T$ is compact, there exist a unique $E_i$-reduced divisor $D_i$ in $T$ and a unique $E_i$-reduced divisor $D'_i$ in $T'$. Then $D_1,D_2,\ldots$ is a Cauchy sequence in $T$ and $D'_1,D'_2,\ldots$ is a Cauchy sequence in $T'$, since $\rho(D_m,D_n)\leqslant\rho(E_m,E_n)$ and $\rho(D'_m,D'_n)\leqslant\rho(E_m,E_n)$ by Lemma\ref{lem:dist_decrease}. Let $D_0\in T$ be the limit of $D_1,D_2,\ldots$ and $D'_0\in T$ be the limit of $D'_1,D'_2,\ldots$. Consider the t-segments $\tconv(D_0,D'_0)$. Then we get another Cauchy sequence $F_1,F_2,\ldots$ in $\tconv(D_0,D'_0)$, where $F_i$ be the $E_i$-reduced divisor in $\tconv(D_0,D'_0)$. If $E_0\in\tconv(D,D_0)$ is the limit of
$F_1,F_2,\ldots$, then we have $$\rho(E_n,E_0)\leqslant\rho(E_n,F_n)+\rho(F_n,E_0)\leqslant\max(\rho(D_n,D_0),\rho(D'_n,D'_0)+\rho(F_n,E_0),$$
where the second inequality follows from Corollary \ref{cor:red_dist_ineq}. Thus $\rho(E_n,E_0)\rightarrow 0$ as $n\rightarrow\infty$ as claimed.

Second, we show that $\tilde{T}$ is totally bounded, i.e., for every real $\epsilon>0$, there exists a finite cover of $\tilde{T}$ by open balls of radius $\epsilon$. We start with a finite cover of $T$ by open balls $B^T(D_i,\epsilon/2)\subseteq T$ of radius $\epsilon/2$ with centers $D_i\in T$ for $i=1,\ldots,n$, and a finite cover of $T'$ by open balls $B^{T'}(D'_j,\epsilon/2)\subseteq T'$ of radius $\epsilon/2$ with centers $D'_j\in T'$ for $j=1,\ldots,m$.
Then for each $\tconv(D_i,D'_j)$, we have a finite cover by open balls $B^{(i,j)}(D^{(i,j)}_{k^{(i,j)}},\epsilon/2)\subseteq\tconv(D_i,D'_j)$ of radius $\epsilon/2$ with the centers $D^{(i,j)}_{k^{(i,j)}}\in \tconv(D_i,D'_j)$ for $k^{(i,j)}=1,\ldots,m^{(i,j)}$. We claim that there is a finite cover of $\tilde{T}$ by open balls $B^{\tilde{T}}(D^{(i,j)}_{k^{(i,j)}},\epsilon)\subseteq\tilde{T}$ of radius $\epsilon$ with the centers $D^{(i,j)}_{k^{(i,j)}}\in\tilde{T}$ for $i=1,\ldots,n$, $j=1,\ldots,m$ and $k^{(i,j)}=1,\ldots,m^{(i,j)}$.
For any $E\in\tilde{T}$, there exist $D\in T$ and $D'\in T'$ such that $E\in\tconv(D,D')$. Suppose $D\in B^T(D_i,\epsilon/2)$ for some $i$ and $D'\in B^{T'}(D'_j,\epsilon/2)$ for some $j$. Furthermore, let $F$ be the $E$-reduced divisor in $\tconv(D_i,D'_j)$ and suppose $F\in B^{(i,j)}(D^{(i,j)}_{k^{(i,j)}},\epsilon/2)$ for some $D^{(i,j)}_{k^{(i,j)}}$. We have
$$\rho(E,D^{(i,j)}_{k^{(i,j)}})\leqslant\rho(E,F)+\rho(F,D^{(i,j)}_{k^{(i,j)}})
\leqslant\max(\rho(D,D_i),\rho(D',D'_j))+\rho(D'',D^{(i,j)}_{k^{(i,j)}})<\epsilon/2+\epsilon/2=\epsilon,$$
where the second inequality follows from Corollary \ref{cor:red_dist_ineq}. Thus $E$ lies in $B^{\tilde{T}}(D^{(i,j)}_{k^{(i,j)}},\epsilon)$, which means $\tilde{T}$ is covered by this finite collection of open balls as claimed.
\end{proof}

\subsection{Finitely generated tropical convex hulls}

Recall that Lemma~\ref{lem:clutching} provides a criterion for judging whether a divisor $D$ lies in a tropical segment $\tconv(D_1,D_2)$, and Lemma~\ref{lem:RD_on_tsegment} extends the criterion. The following theorem generalizes these results to all finitely generated tropical convex hulls, which are compact according to Corollary~\ref{cor:finite_gen}.

\begin{thm} \label{thm:finite_criterion}
Let $T\subseteq\RDivPlusD$ be a tropical convex hull finitely generated by $D_1,\ldots,D_n$. Then for any $E\in\RDivPlusD$, we have $E\in T$ if and only if $\bigcup_{i=1}^n\Gmin(f_{D_i-E})=\Gamma$. Furthermore, if $D_0$ is the $E$-reduced divisor in $T$ and $D$ is an arbitrary divisor in $T$, then
$$\Gmin(f_{D-E})\subseteq\Gmin(f_{D_0-E})=\bigcup_{i=1}^n\Gmin(f_{D_i-E}).$$
\end{thm}
\begin{proof}
We prove by induction on the number of generators. Suppose the statements are true for all tropical convex hulls generated by $n$ divisors. Now consider a tropical convex hull $T$ generated by $n+1$ divisors $D_1,\ldots,D_{n+1}$. Let $T'=\tconv(D_1,\ldots,D_n)$ be a t-convex subset of $T$. For $E\in\RDivPlusD$, let $D_0$ be the $E$-reduced divisor in $T$. By Theorem~\ref{thm:construction}, there exists $D'_0\in T'$ such that $D_0\in\tconv(D'_0,D_n)$, which implies $\Gmin(f_{D'_0-D_0})\bigcup\Gmin(f_{D_n-D_0})=\Gamma$ by Lemma~\ref{lem:clutching}. By assumption, we have
$$\Gmin(f_{D'_0-D_0})\subseteq\bigcup_{i=1}^n\Gmin(f_{D_i-D_0}).$$ Thus, $\bigcup_{i=1}^{n+1}\Gmin(f_{D_i-D_0})=\Gamma$.

In addition, $\Gmin(f_{D_i-D_0})=\Gmin(f_{D_0-E})\bigcap\Gmin(f_{D_i-D_0})$. Therefore,
\begin{align*}
\Gmin(f_{D_0-E})&=\Gmin(f_{D_0-E})\bigcap(\bigcup_{i=1}^{n+1}\Gmin(f_{D_i-D_0})) \\
&= \bigcup_{i=1}^{n+1}(\Gmin(f_{D_0-E})\bigcap\Gmin(f_{D_i-D_0})) \\
&= \bigcup_{i=1}^{n+1}\Gmin(f_{D_i-E}).
\end{align*}
And this also implies $E\in T$ if and only if $\bigcup_{i=1}^{n+1}\Gmin(f_{D_i-E})=\Gamma$.
\end{proof}


Let $T$ be a tropical convex set. For $D\in T$, if $D\notin \tconv(T\setminus\{D\})$, (note that equivalently this means $T\setminus\{D\}$ is also tropically convex) then we say $D$ is an \emph{extremal} of $T$. It is clear from definition that any generating set of $T$ must contain all the extremals of $T$.

\begin{thm} \label{thm:extremal}
Every finitely generated tropical convex hull $T$ contains finitely many extremals. The set $S$ of all extremals of $T$ generates $T$ and is minimal among all generating sets of $T$.
\end{thm}
\begin{proof}
Let $S'$ be a finite generating set of $T$, i.e., $\tconv(S')=T$. We may choose a subset $S$ of $S'$ such that $\tconv(S)=T$ and $S$ is t-convex independent. (The uniqueness of the choice of $S$, which follows from the assertion in the theorem, is not required now.) We claim $S$ is the set of all extremals of $T$, which also implies the minimality of $S$.

Let $S=\{D_0,D_1,D_2,\ldots,D_n\}$ and $T=\tconv(D_1,\ldots,D_n)$. Since $S$ is t-convex independent, we must have $D_0\notin T'$, which implies $\bigcup_{i=1}^n\Gmin(f_{D_i-D_0})\neq\Gamma$ by Theorem~\ref{thm:finite_criterion}. It suffices to show that $D_0$ is an extremal of $T$, i.e., $T\setminus\{D_0\}$ is tropically convex. Choose arbitrarily $E_1$ and $E_2$ in $T\setminus\{D_0\}$. According to Theorem~\ref{thm:construction}, there exist $F_1$ and $F_2$ in $T'$ such that $E_1\in\tconv(D_0,F_1)$ and $E_2\in\tconv(D_0,F_2)$. Note that it follows $\Gmin(f_{E_1-D_0})=\Gmin(f_{F_1-D_0})$ and $\Gmin(f_{E_2-D_0})=\Gmin(f_{F_2-D_0})$. By Theorem~\ref{thm:finite_criterion}, we have
$$\Gmin(f_{F_1-D_0})\subseteq\bigcup_{i=1}^n\Gmin(f_{D_i-D_0})$$ and $$\Gmin(f_{F_2-D_0})\subseteq\bigcup_{i=1}^n\Gmin(f_{D_i-D_0}).$$
Then,
\begin{align*}
\Gmin(f_{E_1-D_0})\bigcup\Gmin(f_{E_2-D_0}) = \Gmin(f_{F_1-D_0})\bigcup\Gmin(f_{F_2-D_0})\subseteq\bigcup_{i=1}^n\Gmin(f_{D_i-D_0})\neq\Gamma,
\end{align*}
which implies $D_0\notin\tconv(E_1,E_2)$. Therefore, $T\setminus\{D_0\}$ is tropically convex as claimed.
\end{proof}

\section{Canonical projections} \label{S:Canproj}
The existence and uniqueness of a reduced divisor in a compact tropical convex set $T$ with respect to an effective $\mR$-divisor of the same degree enable us to define a projection map to $T$.
\begin{dfn} \label{def:RDM}
For a compact tropical convex set $T$ of degree $d$, the \emph{canonical projection} to $T$, $\gamma^T:\RDivPlus\rightarrow T$, is given by sending $E$ to the $E'$-reduced divisor $T_{E'}$ in $T$ where $E'=\frac{d}{\mathrm{deg}E}E$.
\end{dfn}

\begin{lem} \label{lem:RDM_continuous}
Restricted to degree $d$, a reduced-divisor map $\gamma^T|_{\RDivPlusD}$ is continuous.
\end{lem}
\begin{proof}
This is an immediate corollary of Lemma~\ref{lem:dist_decrease} and Lemma~\ref{lem:rescale}.
\end{proof}

\begin{rmk}
For a complete linear system $|D|$, Omini \cite{A12} defined the reduced-divisor map: $\mathrm{Red}:\Gamma\rightarrow|D|$ by sending a point $q\in\Gamma$ to the (conventional) reduced divisor $D_q\in|D|$. In our setting, the map $\mathrm{Red}$ is precisely $\gamma^{|D|}|_{\Div^1_+}$.
\end{rmk}

Let us recall some basic topological notions of retractions and retracts. If $Y$ is a subspace of a topological space $X$, then a \emph{retraction} of $X$ onto $Y$ is a continuous surjection $r:X \twoheadrightarrow Y$ such that $r|_Y=id_Y$. A \emph{deformation retraction} of $X$ onto $Y$ is a homotopy between the identity map of $X$ and a retraction of $X$ onto $Y$, or more explicitly, a continuous map $h:[0,1]\times X \rightarrow X$ such that for all $x\in X$ and $y\in Y$, $h(0,x)=x$, $h(1,x)\in Y$, and $h(1,y)=y$. If in addition $h(t,y)=y$ for all $t\in[0,1]$ and $y\in Y$, then $h$ is called a \emph{strong deformation retraction}. With respect to the existence of a retraction, a deformation retraction or a strong deformation retraction of $X$ onto $Y$, we say $Y$ is a \emph{retract}, a \emph{deformation retract} or a \emph{strong deformation retract} of $X$.

Now let $T\subseteq\RDivPlusD$ be a compact tropical convex set. We know that the canonical projection $\gamma^T|_{\RDivPlusD}$ is continuous (Lemma~\ref{lem:RDM_continuous}) and $\gamma^T|_T=id_T$ (Lemma~\ref{lem:identity}). Therefore, $T$ is a retract of $\RDivPlusD$ with $\gamma^T|_{\RDivPlusD}$ the retraction. In addition, we can use the reduced-divisor map to construct a strong deformation retraction on $T$.

\begin{dfn}
Let $X\subseteq\RDivPlusD$ be tropically convex. Let $T\subseteq X$ be a compact tropical convex subset of $X$. Then we say a strong deformation retraction $h:[0,1]\times X \rightarrow X$ of $X$ onto $T$ is a \emph{tropical retraction} if at each $t\in[0,1]$, the set $h(t,X)$ is tropically convex. In this sense, we say $T$ is a \emph{tropical retract} of $X$.
\end{dfn}

\begin{thm}
For each compact tropical convex subset $T$ of a tropical convex set $X\subseteq\RDivPlusD$, there exists a tropical retraction of $X$ onto $T$.
\end{thm}
\begin{proof}
Our proof will be very similar to the proof of Theorem~\ref{thm:contractible}. We will explicitly construct such a tropical retraction $h:[0,1]\times X \rightarrow X$. In particular, for each $D\in W$, we want $h(0,D)=D$ and $h(1,D)=\gamma^T(D)$.

Let $\rho_T(D):=\min_{D'\in T}\rho(D,D')$. Note that $\rho(D,\gamma^T(D)) = \rho_T(D)$ (Proposition~\ref{prop:levelset_pt}~(1)). Let $\kappa=\sup_{D\in X}\rho_T(D)$. We define $h$ in the following way. For any $D\in W$, we let $h(t,D)=D$ if $t\in[0,1-\frac{\rho_T(D)}{\kappa})$, and $h(t,D)=P_{\gamma^T(D)-D}(\frac{\kappa}{\rho_T(D)}(t-1)+1)$ if $t\in[1-\frac{\rho_T(D)}{\kappa},1]$. In other words, if $\rho_T(D)<\kappa (1-t)$, then $h(t,D)=D$, and otherwise, $h(t,D)$ lies on the t-segment $\tconv(D,\gamma^T(D))$ with distance $\kappa (1-t)$ to $\gamma^T(D)$. It can be easily verified that $h(0,D) = D$ and $h(1,D) = \gamma^T(D)$. In addition, if $D\in T$, then $h(t,D)=D=\gamma^T(D)$ for all $t\in[0,1]$. Now, to show $h$ is actually a tropical retraction of $X$ onto $T$, it remains to show that $h$ is continuous, and $h(t,X)$ is tropically convex for all $t\in[0,1]$.

To say $h$ is continuous is equivalent to say $h(t_n,D_n)\rightarrow h(t,D)$ whenever $t_n\rightarrow t$ and $D_n\rightarrow D$.  We have $$\rho(h(t_n,D_n),h(t,D))\leqslant\rho(h(t_n,D_n),h(t,D_n))+\rho(h(t,D_n),h(t,D)),$$
and $$\rho(h(t_n,D_n),h(t,D_n))\leqslant\rho(D_0,D_n)|t_n-t|\leqslant\kappa|t_n-t|.$$
In stead of proving $\rho(h(t,D_n),h(t,D))\leqslant\rho(D_n,D)$ as in the proof of Theorem~\ref{thm:contractible}, here we claim that $\rho(h(t,D_n),h(t,D))$ is bounded by $2\cdot\rho(D_n,D)$, which is still sufficient to guarantee the continuity of $h$.

Let $h(t,D_n)=D'_n$ and $h(t,D)=D'$. Note that $\gamma^T(D_n)=\gamma^T(D'_n)$ and $\gamma^T(D)=\gamma^T(D')$ (Lemma~\ref{lem:red_fiber}). Denote these reduced divisors by $C_n$ and $C$ respectively. Also, we note that $\rho(D_n,D)\geqslant\rho(C_n,C)$ (Lemma~\ref{lem:dist_decrease}).

Case (1): $\rho_T(D_n)<\kappa (1-t)$ and $\rho_T(D)<\kappa (1-t)$. Then $D'_n=D_n$ and $D'=D$. We automatically have $\rho(D'_n,D')=\rho(D_n,D)$.

Case (2): $\rho_T(D_n)<\kappa (1-t)$ and $\rho_T(D)\geqslant \kappa (1-t)$.  Then $D'_n=D_n$ and $$\rho(D_n,C_n)=\rho_T(D_n)<\rho(D',C)=\rho_T(D')=\kappa (1-t).$$
Let $D''\in\tconv(C,D)$ be the $D_n$-reduced divisor in $\tconv(C,D)$. Depending on the relative positions of $D'$ and $D''$ in $\tconv(C,D)$, there are two subcases.

Subcase (2a): $D''\in\tconv(C,D')$. By Proposition~\ref{prop:dist_ineq}, we have $$\rho(D'_n,D')=\rho(D_n,D')\leqslant\rho(D_n,D).$$

Subcase (2b): $D'\in\tconv(C,D'')$. Now let $C''$ be the $C_n$-reduced divisor in $\tconv(C,D)$. Then we have
$\rho(C'',C)\leqslant\rho(C_n,C)$ and $\rho(D'',C'')\leqslant\rho(D_n,C_n)$ (Lemma~\ref{lem:dist_decrease}). Also, we note that $\rho(D_n,C_n)<\rho(D',C)=\kappa (1-t)$.
Therefore,
\begin{align*}
\rho(D'',D') &=\rho(D'',C)-\rho(D',C)\leqslant(\rho(D'',C'')+\rho(C'',C))-\rho(D',C) \\
&\leqslant(\rho(D_n,C_n)+\rho(C_n,C))-\rho(D',C)<\rho(C_n,C)\leqslant\rho(D_n,D).
\end{align*}
Moreover, by Corollary~\ref{cor:red_dist_ineq}, we have $$\rho(D_n,D'')\leqslant\max(\rho(C_n,C),\rho(D_n,D))=\rho(D_n,D).$$
It follows $$\rho(D'_n,D')=\rho(D_n,D')\leqslant\rho(D_n,D'')+\rho(D'',D')<2\cdot\rho(D_n,D).$$

Case (3): $\rho_T(D_n)\geqslant\kappa (1-t)$ and $\rho_T(D)<\kappa (1-t)$. Then $D'=D$ and $\rho_T(D)<\rho_T(D'_n)=\kappa (1-t)$. Exchanging the roles of $D_n$ and $D$, we may analyze this case in the same way as in Case~(2), and conclude that $\rho(D'_n,D')<2\cdot\rho(D_n,D)$ in general.

Case (4): $\rho_T(D_n)\geqslant\kappa (1-t)$ and $\rho_T(D)\geqslant\kappa (1-t)$. In this case,
$$\rho(D'_n,C_n)=\rho_T(D_n)=\rho(D',C)=\rho_T(D')=\kappa (1-t).$$
Let $D''\in\tconv(C,D)$ be the $D'_n$-reduced divisor in $\tconv(C,D)$ and $D''_n\in\tconv(C_n,D_n)$ be the $D'$-reduced divisor in $\tconv(C_n,D_n)$. We need to consider the relative the positions of $D'$ and $D''$ in $\tconv(C,D)$ and the relative positions of $D'_n$ and $D''_n$ in $\tconv(C_n,D_n)$.

Case (4a): $D''\in\tconv(C,D')$ and $D''_n\in\tconv(C_n,D'_n)$. Then we can apply Proposition~\ref{prop:dist_ineq} and see that $\rho(D'_n,D')\leqslant\rho(D_n,D)$.

Case (4b): $D''\in\tconv(D',D)$ and $D''_n\in\tconv(D'_n,D_n)$. Again we can apply Proposition~\ref{prop:dist_ineq} and get $\rho(D'_n,D')\leqslant\rho(C_n,C)\leqslant\rho(D_n,D)$.

Case (4c): $D''\in\tconv(D',D)$ and $D''_n\in\tconv(C_n,D'_n)$. We can use an analogous analysis as in Case~(2b) and get $$\rho(D'_n,D')\leqslant\rho(D'_n,D'')+\rho(D'',D')\leqslant2\cdot\rho(D_n,D).$$ Note that we get $``\leqslant''$ instead of $``<''$ as in Case~(2b) because we now have $\rho(D'_n,C_n)=\rho(D',C)=\kappa (1-t)$.

Case (4d): $D''\in\tconv(C,D')$ and $D''_n\in\tconv(D'_n,D_n)$. Base on a similar analysis as in Case~(4c), we get
$$\rho(D'_n,D')\leqslant\rho(D'_n,D''_n)+\rho(D''_n,D')\leqslant2\cdot\rho(D_n,D).$$

So far we've finished the proof of the continuity of $h$. To show $h(t,X)$ is tropically convex, we note that $h(t,T)$ is the sublevel set $L_{\leqslant r}^{T}(\rho_T)$ of the distance function $\rho_T$ where $r=\kappa (1-t)$. Hence we only need to show that choosing arbitrarily $D_1$ and $D_2$ from $X$ such that $\rho_T(D_1)\leqslant r$ and $\rho_T(D_2)\leqslant r$, we must have $\rho_T(D)\leqslant r$ for every $D\in\tconv(D_1,D_2)$. Let $C_1$ and $C_2$ the reduced divisors in $T$ with respect to $D_1$, $D_2$ respectively. Let $C$ be the $D$-reduced divisor in $\tconv(C_1,C_2)$. Then by Corollary~\ref{cor:red_dist_ineq}, we must have
$$\rho_T(D)\leqslant\rho(D,C)\leqslant\max(\rho(D_1,C_1),\rho(D_2,C_2))=\max(\rho_T(D_1),\rho_T(D_2))\leqslant r.$$
\end{proof}


\end{document}